\documentclass[11pt]{amsart}
\usepackage{amsmath, amssymb, esint}
\usepackage{bbm}
\usepackage{latexsym}
\usepackage{xcolor}
\usepackage{tikz}
\usetikzlibrary{arrows}
\usetikzlibrary{patterns}
\usetikzlibrary{shapes}
\usepackage{mathrsfs}
\usepackage{bbm}
\usepackage{amsfonts}

 \usepackage[usenames,dvipsnames]{pstricks}
 \usepackage{epsfig}
 \usepackage{pst-grad} 
 \usepackage{pst-plot} 

\usepackage{tikz-cd}
\newcommand{\C}{\mathbb{C}}

\newcommand{\R}{\mathbb{R}}

\newcommand{\EE}{\mathbb{E}}
\newcommand{\PP}{\mathbb{P}}
\newcommand{\ol}{\overline}

\newcommand{\e}{\varepsilon}

\newtheorem{thm}{Theorem}
\newtheorem{lemma}[thm]{Lemma}
\newtheorem{cor}[thm]{Corollary}

\newtheorem{prop}[thm]{Proposition}

\newtheorem*{thmi}{Theorem}

\newtheorem*{cori}{Corollary}

\theoremstyle{remark} 
\newtheorem{remark}[]{Remark}

\newcommand{\be}{\begin{equation}}
\newcommand{\ee}{\end{equation}}

\usepackage{mathtools} 
\mathtoolsset{showonlyrefs, showmanualtags} 

\title[On the geometry of random lemniscates]{On the geometry of random lemniscates}
\author{Antonio Lerario}   \thanks{Antonio Lerario, SISSA (Trieste). email: \textsf{lerario@sissa.it}.}
\author{Erik Lundberg}   \thanks{Erik Lundberg, Florida Atlantic University, Department of Mathematical Sciences. email: \textsf{elundber@fau.edu}.}

\begin{document}

\begin{abstract}
We investigate the geometry of a random rational lemniscate $\Gamma$, the level set $\{|r(z)|=1\}$ on the Riemann sphere $\hat\C=\C\cup\{\infty\}$ of the modulus of a random rational function $r$. 
We assign a probability distribution to the space of rational functions $r=p/q$ of degree $n$ by sampling $p$ and $q$ independently from the complex Kostlan ensemble
of random polynomials of degree $n$.

We prove that the average \emph{spherical length} of $\Gamma$ is 
$\frac{\pi^2}{2} \sqrt{n},$
which is proportional to the square root of the maximal spherical length.
We also provide an asymptotic for the average number
of points on the curve that are tangent to one of the
meridians on the Riemann sphere (i.e. tangent to one of the radial directions in the plane).

Concerning the topology of $\Gamma$, on a local scale, we prove that for every disk $D$ of radius $O(n^{-1/2})$ in the Riemann sphere  and any \emph{arrangement} 
(i.e. embedding) of finitely many circles $A\subset D$ there is a positive probability 
(independent of $n$) that $(D,\Gamma\cap D)$ is isotopic to $( D,A)$. (A local random version of Hilbert's Sixteenth Problem restricted to lemniscates.)
Corollary: the average number of connected components of $\Gamma$ increases linearly (the maximum rate possible according to a deterministic upper bound).
\end{abstract}

\maketitle

\section{Introduction}

The current paper investigates the geometry and topology of \emph{random} rational lemniscates, a theme that offers a complementary viewpoint (namely, seeking the typical outcome as opposed to the extremal outcome) 
to many classical studies and quickly leads to open problems that are simple to state.

\subsection{Lemniscates}
A rational lemniscate $\Gamma$ is the level set $\{z\in \C : |r(z)| = t\}$ of the modulus of a rational function $r(z)$.
For a generic rational function $r(z)$ of degree $n$, 
we can express this in terms of two complex polynomials $p$ and $q$ of degree $n$
\be\label{eq:lemni} \Gamma=\left\{z\in \C\,:\,\,\left|\frac{p(z)}{q(z)}\right|=1\right\}.\ee
Alternatively, writing the defining equation of $\Gamma$ as $p(z) \ol{p(z)} - q(z) \ol{q(z)} = 0$
and setting $z=x+i y$, we see that a lemniscate of degree $n$ is a real algebraic curve of degree $2n$.
In fact, some of the classical examples of algebraic curves are lemniscates \cite[p. 120-124]{Lawrence}.
From yet another point of view, 
lemniscates have been studied extensively in logarithmic potential theory \cite{PoulRans, SaffTotik};
notice that $\Gamma$ can be viewed as the zero set of the logarithmic potential
$\log|p(z)| - \log|q(z)|$ generated by $n$ positive charges (positioned at the zeros of $p$) and $n$ negative charges (positioned at the zeros of $q$).

Lemniscates appear in a variety of specific studies and applications
including approximation theory (e.g. Hilbert's lemniscate theorem and its generalizations \cite{Walsh, NagyTotik}),
topology of real algebraic curves \cite{Catanese2, Catanese1, Kharlamov},
elliptic integrals from classical mechanics \cite{Ayoub},
holomorphic dynamics \cite[p. 151]{Milnor}, 
numerical analysis \cite{TrefBau},
operator theory \cite{TrefEmbree},
so-called ``fingerprints'' of two-dimensional shapes \cite{EbenKhSh, Younsi},
moving boundary problems \cite{KhavMPT, LundTotik},
as critical sets of planar harmonic mappings \cite{KhavLS, LLtrunc},
and in the theory and applications of conformal mapping \cite{Bell, GustPSS, JeongTan}.

For a generic rational function $r(z) = p(z)/q(z)$
its degree $n$ determines the number of zeros and poles.
On the other hand, the geometry and topology of the level set $\Gamma$ is more diverse.
For example, the maximal (spherical) length of $\Gamma$ is $2 \pi n$ \cite[Theorem 2]{EremenkoHayman}, 
but any outcome in the interval $(0,2\pi n)$ can be realized. 
Concerning the topology of $\Gamma$, it is generically a smooth curve that can consist of at most $n$ topological circles (see Proposition \ref{prop:components} below), 
and all possible configurations of these circles in the plane (up to isotopy) can be realized (see Proposition \ref{prop:H16}).

\begin{figure}
\begin{center}
\includegraphics[width=0.27\textwidth]{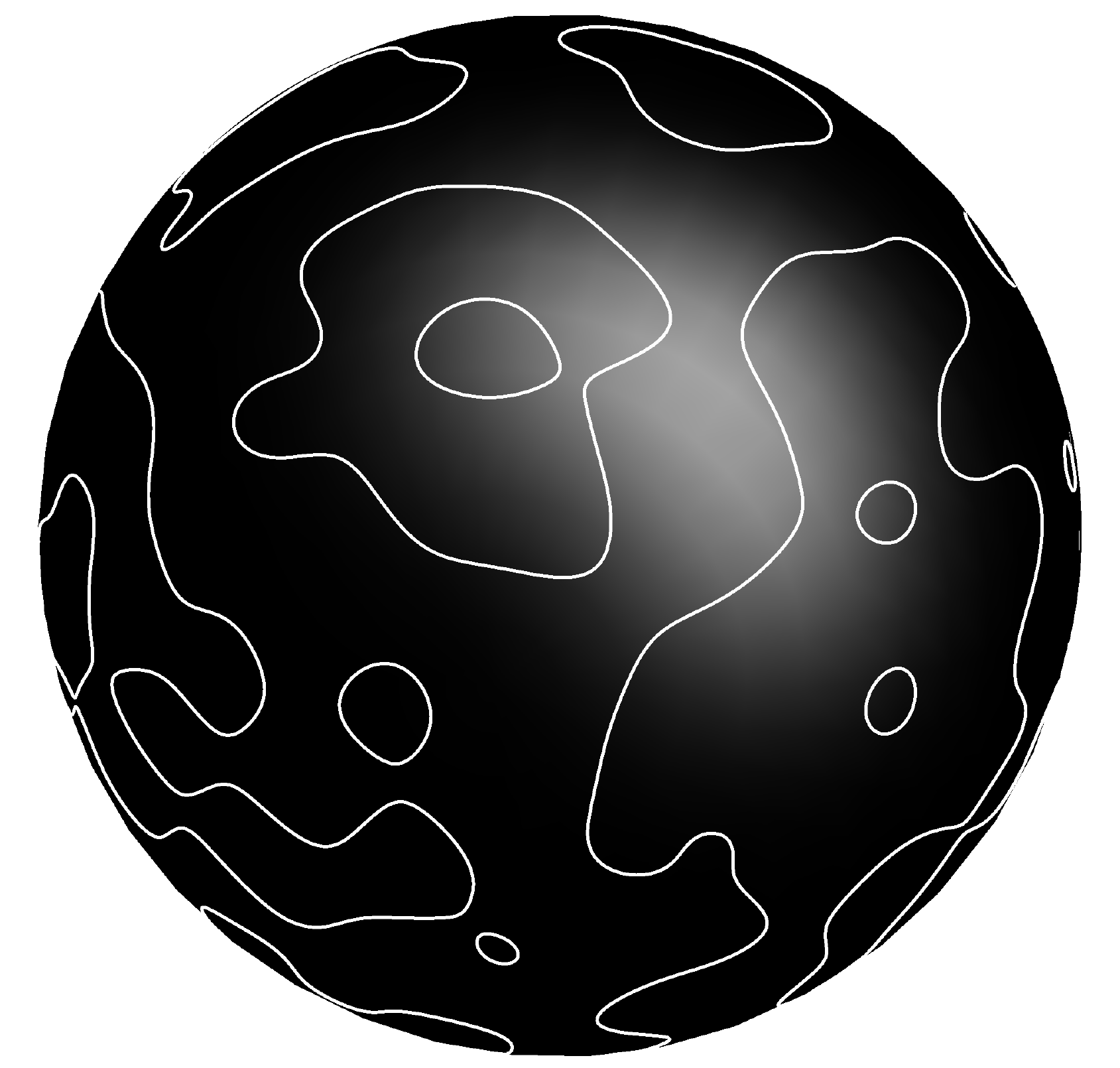} \hspace{.2in}
\includegraphics[width=0.27\textwidth]{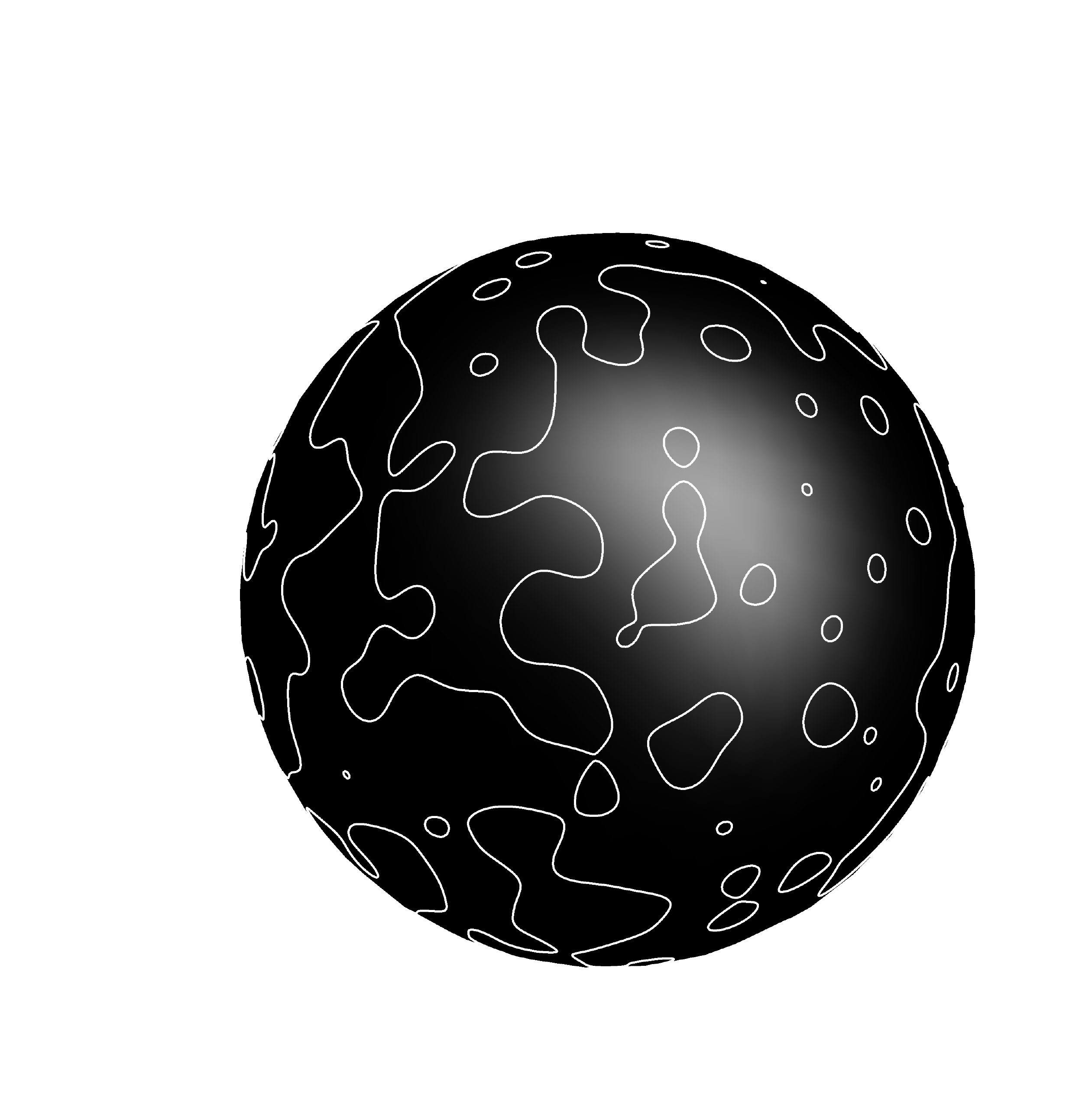}\hspace{.2in}
\includegraphics[width=0.27\textwidth]{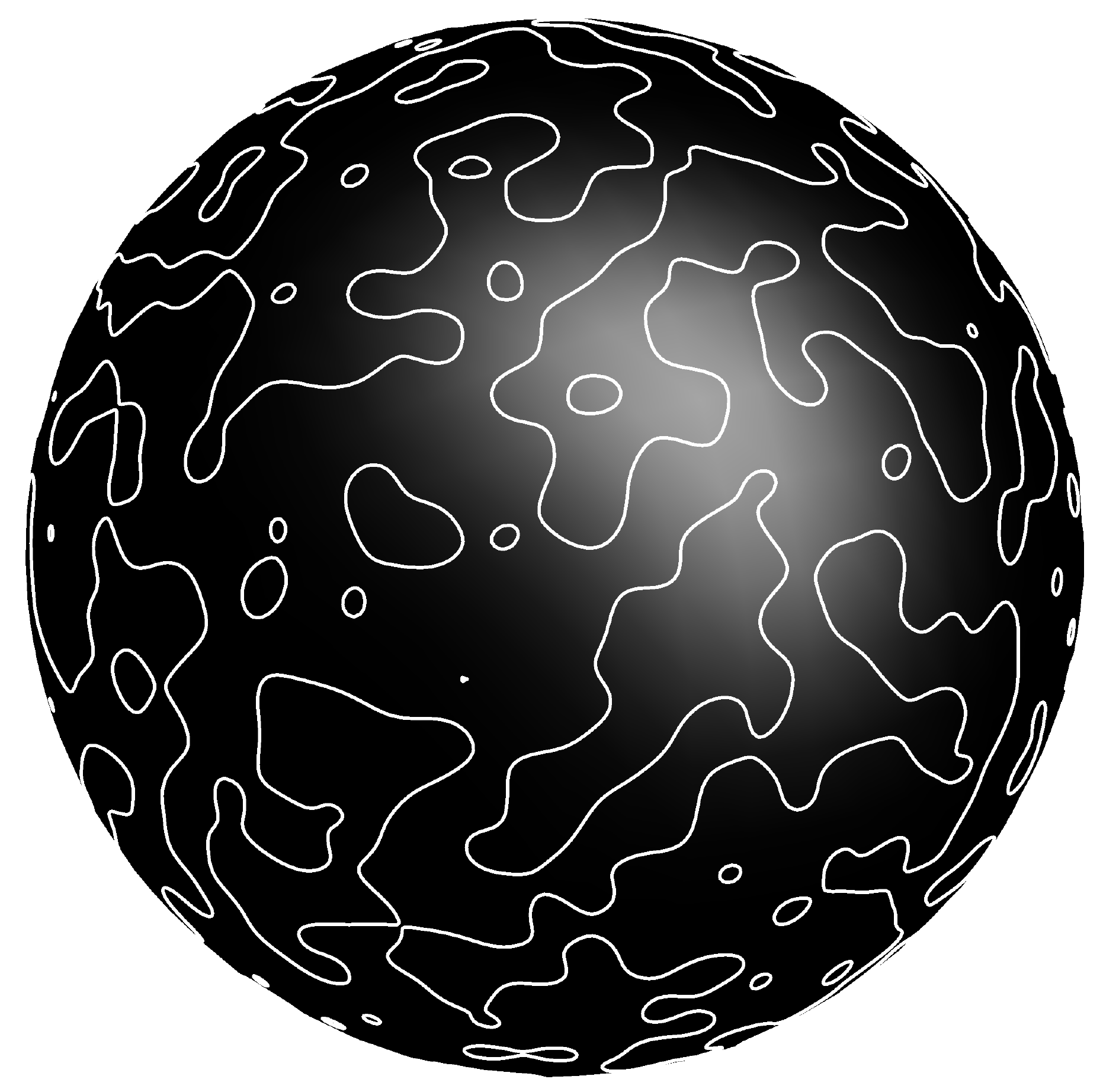}
\caption{Pictures of random lemniscates of degree $n=100,200,400$}\label{fig:Lemniscates}
\end{center}
\end{figure}

\subsection{Random lemniscates}
In this paper, motivated by the high level of variability in the geometry and topology of rational lemniscates,
 we investigate the \emph{average} outcome of certain properties, 
while considering the defining polynomials to be random complex \emph{Kostlan} polynomials (also called \emph{elliptic} polynomials, or \emph{Fubini-Study} polynomials). 
Namely, in \eqref{eq:lemni} we let:
\be \label{eq:pq}
p(z)=\sum_{k=0}^na_k z^k,\quad \textrm{and}\quad q(z)=\sum_{k=0}^nb_k z^k,
\ee
where the coefficients $a_k$ and $b_k$, are independent centered \emph{complex} Gaussians with:
\be \EE a_j\overline a_k=\delta_{jk}{n \choose k}\quad \textrm{and}\quad  \EE b_j\overline b_k=\delta_{jk}{n \choose k}.\ee
One important feature of this probability distribution is that the \emph{random} rational function $p/q$ is invariant under special unitary M\"obius transformations (see Lemma \ref{lemma:invariance} below):
\be \frac{p(z)}{q(z)}\sim \frac{p\left(\frac{\lambda z+\mu}{-\overline \mu z+\overline \lambda}\right)}{q\left(\frac{\lambda z+\mu}{-\overline \mu z+\overline \lambda}\right)}\quad \textrm{as random rational functions, if $\lambda\overline \lambda+\mu\overline\mu=1$}.\ee
\begin{remark}The poles and the zeros of $p/q$ are independent and distributed in the same way;  using the explicit form of this distribution, S. Zelditch and O. Zeitouni \cite{ZeZe} have derived a large deviation principle for the empirical measure associated to the zeros of a random complex Kostlan polynomial.\end{remark}

\subsection{The average length}
A. Eremenko and W. Hayman \cite{EremenkoHayman} have proved that the maximal length of $\Gamma$ 
(measured with respect to the spherical metric induced by the stereographic projection $\sigma:S^2\to \hat\C$) is $2\pi n$, 
the extremal case being given by the rational function $r(z)=z^n$. 
The following theorem is proved below for the average length $\EE |\Gamma|$ in our random model.
It is interesting that the average outcome turns out to have a simple answer\footnote{The known outcome for the Kostlan random curve is also simple, see Figure \ref{fig:comparisonaverage}.}
considering that 
expressing the length of \emph{specific} lemniscates can require elliptic integrals.

\begin{thmi}\label{thmi:length}The average spherical length of $\Gamma$ is $\frac{\pi^2}{2} \sqrt{n}.$
\end{thmi}

The proof is given in Section \ref{sec:length} and consists of two parts.
The unitary invariance of the model corresponds to invariance under rotation of the Riemann sphere; 
using this fact we apply the integral geometry formula in order to 
reduce the computation of the expectation of the length to the evaluation of a Kac-Rice integral 
(equations \eqref{eq:IGF} and \eqref{eq:length} below).
This ultimately reduces to a one-dimensional integral:
\be\label{eq:KRintro}
\EE |\Gamma|= 2\pi^2 \int_{\R}|x_2|\rho(0,x_2)dx_2,
\ee
where $\rho(x_1,x_2)$ denotes the joint probability density function of the two random variables:
\be X_1=a_0 \ol{a_0} - b_0 \ol{b_0}\quad \textrm{and}\quad X_2= \frac{1}{2}\left(a_1 \ol{a_0} + a_0 \ol{a_1} - b_1 \ol{b_0} - b_0 \ol{b_1} \right),\ee
where $a_0,a_1,b_0,b_1$ are the random coefficients from \eqref{eq:pq} above.

The second part of the proof, an exact evaluation of the integral in \eqref{eq:KRintro}, is performed using the method of characteristic functions (i.e. Fourier transformation):
\be \rho(x_1,x_2) = \frac{1}{(2\pi)^2} \int_{\R^2} e^{-i(sx_1 + tx_2)} \hat{\rho}(s,t) ds dt,\quad\textrm{where}\quad \hat{\rho}(s,t) = \EE e^{i(sX_1 + tX_2)}.
\ee
Using the fact that $sX_1+tX_2$ is a \emph{Hermitian} quadratic form in Gaussian variables, we can obtain explicitly the characteristic function $\hat{\rho}(s,t)= \frac{1}{(1+nt^2/4)^2 + s^2}$.
The integration for recovering $\rho(0,x_2)=\frac{1}{2\sqrt{n}}e^{-2|x_2|/\sqrt{n}}$ is carried out using residues,
and the evaluation of \eqref{eq:KRintro} is then elementary.

This explicit computation of \eqref{eq:KRintro} serves as a paradigm for 
analyzing a finer geometric 
property of $\Gamma$, its number of ``meridian tangents''.

\subsection{Meridian tangents}Before moving to the study of the topology of random lemniscates, 
we investigate another geometric feature\footnote{For the nodal set of random plane waves, a similar quantity, the average number of vertical tangents, has been computed by M. Krishnapur, to which the authors are grateful for having shared with them his preprint.}, the number $\nu(\Gamma)$ of times that $\Gamma$ is tangent to some real line through the origin in $\C$. 
Working in spherical coordinates $(\theta, \phi)$ on the Riemann sphere, $\nu(\Gamma)$ equals the number of critical points of the function $\theta|_\Gamma:\Gamma\to S^1$, 
and thus measures the number of times $\Gamma$ ``changes direction'' with respect to the meridian field. 
(The projection to $\theta$ is undefined at the origin and at infinity, but the probability that $\Gamma$ passes through
one of these points is zero.)

\begin{thmi}\label{thmi:meridian}The expectation $\EE \, \nu(\Gamma)$ of the number of meridian tangents of a random lemniscate $\Gamma$ 
is asymptotic to $\left( \frac{32-\sqrt{2}}{28} \right) n.$
\end{thmi}

The number $\nu(\Gamma)$ can also be used to estimate the number of components of $\Gamma$, using Morse inequalities.
Namely,
\be b_0(\Gamma) \leq \frac{\nu(\Gamma)}{2}+\#\{\textrm{components that loop around the origin}\}.\ee

In this context, 
the precise evaluation of the constant in the above theorem is important,
since the resulting upper bound has the same linear growth rate as the deterministic 
statement $b_0(\Gamma) \leq n$, but with a sharper constant,
see Section \ref{isec:H16} below 
(the number of components that loop around the origin
turns out to be $O(\sqrt{n})$, a lower order term).

Using rotational invariance the computation of $\EE \nu(\Gamma)$ is reduced again to a Kac-Rice type integral 
(this time for a system of equations):
\be \EE\nu(\Gamma)=4\pi\int_{\R^2}|h_1h_2|\rho(0,0,h_1,h_2)dh_1\, dh_2,\ee
where $\rho(0,0,h_1, h_2)$ denotes now the joint density of four random variables, each of which is a Hermitian quadratic form in Gaussian variables,
note the analogy with \eqref{eq:KRintro}.

The detailed form of the above integral is displayed in equation \eqref{eq:explicit} below,
and its exact evaluation is performed in Section \ref{sec:Fouriermeridian}.
While the procedure follows the same general outline as in the computation of the average length,
the method of characteristic functions is much more difficult to execute in the case of meridian tangents.

\subsection{Hilbert's Sixteenth Problem for random lemniscates}\label{isec:H16}
A rational lemniscate of degree $n$ is a special \emph{real} algebraic curve of degree $2n$ (whose defining equation is $|p|^2-|q|^2=0$).
The first part of Hilbert's sixteenth problem concerns the topology of real algebraic curves and their embeddings.
The problem has a long and fascinating history and is far from complete in its general form.
It is known that a smooth plane curve of degree $n$ can have at most $\frac{(n-1)(n-2)}{2}+1$ components (Harnack's bound\footnote{This is the bound for curves in the real \emph{projective} plane; 
the bound for curves in the affine plane has the same order $O(n^2).$}, see \cite{BCR}), 
but in the special case of a rational lemniscate we have indeed (see Proposition \ref{prop:components}):
\be\label{eq:introupper} b_0(\Gamma)\leq n.\ee
On the other hand, concerning the \emph{arrangement} of the components in the plane
all possibilities can occur. 
We prove this result in Proposition \ref{prop:H16}; 
it is well-known among experts, but we were unable to find an appropriate reference.

From the probabilistic point of view, this deterministic statement has a ``local'' version provided by the following theorem. 
We say that two pairs of manifolds $(B_1, A_1)$ and $(B_2, A_2)$ are isotopic if there exists a homeomorphism $\psi:B_2\to B_1$ that restricts to a homeomorphism $\psi|_{A_1}:A_1\to A_2$ (i.e. the two embeddings $A_1\hookrightarrow B_1$ and $A_2\hookrightarrow B_2$ look the same). 
\begin{thmi}For every disk $D$ of radius $O(n^{-1/2})$ in the Riemann sphere $\hat\C=\C\cup\{\infty\}$ and any \emph{arrangement} (i.e. embedding) of finitely many circles $A\hookrightarrow D$ 
there is a positive probability (independent of $n$) that $( D, \Gamma\cap D)$ is isotopic to $( D, A)$.
\end{thmi}

The proof is inspired by the so called \emph{barrier method}, 
introduced by F. Nazarov and M. Sodin \cite{NazarovSodin} for the study of nodal sets of random spherical harmonics
and extended in \cite{LLstatistics, GaWe3, FLL} to the study of random algebraic hypersurfaces.
The current application uses transversality from differential topology combined with Markov's inequality from probability.
After rescaling the variable  $z\mapsto n^{-1/2}z$,
we show that with some positive probability (independent of $n$) the prescribed arrangement is not only
realized by a fixed number of terms but with a transversality that is stable with respect to the random perturbation represented by the remaining terms.
Showing that (with some positive probability) the perturbation is sufficiently controlled 
is the crucial step that relies on Markov's inequality.
This adaptation of the barrier method has some similarities with the one used to study real Kostlan curves in \cite{GaWe3}.
However, in the context of the random lemniscate model, 
the rate of scaling for the localized arrangement leads a maximal rate of growth (as $n \rightarrow \infty$)
for the global topological complexity; this follows from the fact that the statement in the corollary below
has the same order of growth as the deterministic upper bound.


We note that the above theorem implies in particular that the average number of components of $\Gamma$ is bounded below by $c_1n$ for some $c_1>0$:
\be\label{eq:introlower} \EE b_0(\Gamma)\geq c_1 n.\ee
It is in fact enough to consider the arrangement $S^1\hookrightarrow D$, 
which appears with some positive probability $c>0$ in a neighborhood with radius of order $n^{-1/2}$: 
covering the sphere with order $n$ many disjoint such neighborhoods provides on average order $n$ many components.
Combining \eqref{eq:introlower} with an upper bound (either the deterministic one or the improvement coming from the average number of meridian tangents), 
we derive the following corollary (see Section \ref{sec:components} for a more detailed statement and proof).

\begin{cori}The random lemniscate $\Gamma$ has on average order $n$ many components.
\end{cori}

\subsection{Random curves}
\begin{figure}
$$
\begin{array}{ccc}
\cline{2-3}
\multicolumn{1}{c}{}&\multicolumn{1}{|c|}{}&\multicolumn{1}{|c|}{}\\

&\multicolumn{1}{|c|}{\textrm{random lemniscate}}&\multicolumn{1}{|c|}{\textrm{Kostlan random curve}}\\
\multicolumn{1}{c}{}&\multicolumn{1}{|c|}{}&\multicolumn{1}{|c|}{}\\

\hline
\multicolumn{1}{|c|}{}&\multicolumn{1}{|c|}{}&\multicolumn{1}{|c|}{}\\

\multicolumn{1}{|c|}{\textrm{length}}&\multicolumn{1}{|c|}{\frac{\pi^2}{2}\sqrt{n}}&\multicolumn{1}{|c|}{2\pi \sqrt{n}}\\
\multicolumn{1}{|c|}{}&\multicolumn{1}{|c|}{}&\multicolumn{1}{|c|}{}\\

\hline
\multicolumn{1}{|c|}{}&\multicolumn{1}{|c|}{}&\multicolumn{1}{|c|}{}\\

\multicolumn{1}{|c|}{\textrm{meridian tangents}}&\multicolumn{1}{|c|}{\sim \frac{32-\sqrt{2}}{28}n}&\multicolumn{1}{|c|}{\sim\frac{4\sqrt{2} }{\pi}n}\\
\multicolumn{1}{|c|}{}&\multicolumn{1}{|c|}{}&\multicolumn{1}{|c|}{}\\

\hline
\multicolumn{1}{|c|}{}&\multicolumn{1}{|c|}{}&\multicolumn{1}{|c|}{}\\

\multicolumn{1}{|c|}{\textrm{number of components}}&\multicolumn{1}{|c|}{\Theta(n)}&\multicolumn{1}{|c|}{\Theta(n)}\\
\multicolumn{1}{|c|}{}&\multicolumn{1}{|c|}{}&\multicolumn{1}{|c|}{}\\
\hline
\end{array}
$$
\caption{A comparison of the average outcomes for 
the random lemniscate and the Kostlan random curve.}
\label{fig:comparisonaverage}
\end{figure}
The study of random real algebraic varieties and their topology
has become very active recently, 
having interactions with Statistical Physics and Random Matrix Theory \cite{BogSch, FLL, GaWe2, GaWe3, Lerario2012, LeLu2, LLstatistics, NazarovSodin, NazarovSodin2, Sarnak, SarnakWigman}. 



Previous studies have considered several models of random algebraic curves.
Most attention has been focused on the \emph{real Fubini-Study} model \cite{FLL, LLstatistics}, 
the \emph{complex Fubini-Study} model \cite{GaWe2, GaWe3}
(also known as the \emph{Kostlan} model),
and random spherical harmonics \cite{NazarovSodin} 
(whose nodal sets comprise a special class of algebraic curves). 
Each of these models is a \emph{Gaussian} model (where the defining polynomials have Gaussian coefficients).
A random curve of degree $n$ sampled from the
real Fubini-Study model has on average order $n^2$ many components, while the average number of components for the Kostlan model has order $n$.
The random lemniscate model studied in this paper has some similarities
to the Kostlan model, see Figure \ref{fig:comparisonaverage} and Section \ref{sec:Kostlan} below.

A principal challenge in studying the random lemniscate model is that it is not Gaussian, 
but rather the coefficients are Hermitian forms in complex Gaussian variables.
Another interesting non-Gaussian random curve model is the \emph{determinantal} model discussed in \cite{LeLu2}: 
the defining polynomial  is the restriction to the unit sphere $S^2\subset \R^3$ of the polynomial $p(x,y,z)=\det (xQ_1+y Q_2+z Q_3)$, where $Q_1, Q_2, Q_3$ are random GOE$(n)$ matrices.
The coefficients are products of $n$-many Gaussians; very little is known on this model.

Nazarov and Sodin \cite{NazarovSodin2} have developed general results that 
establish the existence of asymptotic laws for the number of components (in addition to the order of growth).
It would be interesting if this approach could be adapted to non-Gaussian models such as the random lemniscate model.

\subsection{Related problems}
In this paper, we have considered the average spherical length of a rational lemniscate.
The alternative problem of determining the average \emph{planar} length of a \emph{polynomial} lemniscate will be considered in a forthcoming
joint work with Koushik Ramachandran.
Such is motivated by seeking a broad point of view on the \emph{Erd\"os lemniscate problem} \cite{EHP, Erdos} that asks for the maximal planar length of a monic polynomial lemniscate.
The conjectured \cite{EHP, Erdos} extremal case $\{|z^n-1| = 1 \}$ has recently been shown by A. Fryntov and F. Nazarov \cite{FryNaz} to be locally extremal.
They also confirmed that as $n \rightarrow \infty$ the global extremal length is $2 \pi n + o(n)$, which is asymptotic to the conjectured result.

Concerning the topology of lemniscates,
an attractive direction is to investigate the typical \emph{Morsification} of the modulus of a random rational function (or polynomial).
This requires studying the whole one-parameter family of lemniscates $\{ |r(z)| = t \}$ and the arrangement of the singular levels
that are encountered as $t$ varies.
Although a probabilistic study on this topic seems rather ambitious at this stage,
deterministic studies \cite{Catanese1, Catanese2} have already provided complete classifications
including combinatorial schemes for enumerating generic Morsifications (both in the case of polynomials and rational functions).

A slightly different model of random lemniscates arises in the study of random harmonic polynomials.
A (complex) harmonic polynomial is a polynomial of the form $F(z) = p(z) + \ol{q(z)},$ where $p$ and $q$ are analytic polynomials.
The critical set (where the Jacobian determinant vanishes) is a rational lemniscate (the ``critical lemniscate'' associated to $F$):
\be \left\{z\in \C\,:\,\,\left|\frac{p'(z)}{q'(z)}\right|=1\right\}.\ee
Critical lemniscates have been studied recently in \cite{KhavLS},
and in \cite{LLtrunc} we posed the problem of studying the average number of components of the critical lemniscate associated to a random harmonic polynomial.
In the case when $p$ and $q$ are Kostlan polynomials as in \cite{LLtrunc} of the same degree $n$,
we conjecture the same outcome as in the above corollary---that the average number of components of a critical lemniscate grows linearly with $n$.
Another object of interest is the image under $F$ of the critical set, referred to as the \emph{caustic}.
The caustic generically has cusp singularities,
and it would be interesting to determine the average number of cusps on the caustic associated to a random harmonic polynomial $F$.

\subsection{Structure of the paper}In Section \ref{sec:preliminaries} we introduce some preliminary tools. 
In Section \ref{sec:length} we prove the statement on the average length of random lemniscates. The average number of meridian tangents is computed in Section \ref{sec:meridian}. 
Section \ref{sec:topology} is devoted to Topology: the deterministic study of the possible arrangements of the components of a lemniscate is discussed in Section \ref{sec:H16} and the local probabilistic study in Section \ref{sec:local}; 
the application to the average number of components is in Section \ref{sec:components}.
For comparison the Kostlan model is briefly discussed in Section \ref{sec:Kostlan} (see also the table in Figure \ref{fig:comparisonaverage}).

\subsection*{Acknowledgements}
We wish to thank  A. Eremenko and V. Kharlamov for useful comments
and M. Krishnapur for sharing some unpublished notes on random plane waves.
We also thank Chelsey Hoff for assistance in generating the graphics
displayed in Figures \ref{fig:Lemniscates} and \ref{fig:LemKost}.
\section{Preliminaries}\label{sec:preliminaries}
\subsection{Kac-Rice type formulas}
An essential ingredient for our study is the \emph{Kac-Rice formula}. 
In a general setting: we want to compute the number of solutions of a system of random polynomial equations $F(z)=(f_1(z), \ldots, f_k(z))=0$, where $z\in\R^k$ (in our cases $z\in \C$ or $z\in S^2)$. 
We denote by $\rho(u, v; z)$ the joint density of the random variable $(F(z),JF(z))$, i.e. for any measurable subset $U\subset \R^k\times \R^{k\times k}$:
\be \PP\{(F(z), JF(z))\in U \}=\int_{U}\rho(u, v; z)dudv,\quad u\in \R^k, v\in \R^{k\times k}.\ee
Then the Kac-Rice formula \cite[Theorem 11.2.1]{AdlerTaylor} asserts that
for a region $D$:
\be\label{eq:KR} \EE\#\{z\in D\,|\,F(z)=0\}=\int_{D}\int_{\R^{k\times k}}|\det(v)|p(0, v;z)dv dz \ee
(we refer to the $dz$ integrand as the \emph{Kac-Rice density}).
In order to apply \cite[Theorem 11.2.1]{AdlerTaylor} some regularity assumptions are required on the random field $F$; it easy to check that the cases of our interest these assumptions are always satisfied. 
In the sequel we will need the following elementary Lemma.
\begin{lemma}\label{lemma:condi}Let $f:\R^2\to \R$ be a random field and set $F=(f, \partial_x f):\R^2\to \R^2$. Assume that $F$ satisfies the assumptions of  \cite[Theorem 11.2.1]{AdlerTaylor}. 
Denote by $\rho(x, y, h_1, h_2; z)$ the joint density of the random vector:
\be (X, Y, H_1, H_2)=(f(z), \partial_x f(z), \partial_y f(z), \partial_{x}^2 f(z)).\ee
Then the Kac-Rice density at $z\in \R^2$ for the system $\{f=\partial_x f=0\}$ can be written as:
\be k(z)=\int_{\R^2}|h_1h_2|\rho(0,0, h_1, h_2;z)dh_1\,dh_2.
\ee
\end{lemma}
\begin{proof}Denote by $P(x,y, g_1, g_2, g_3, g_4;z)$ the joint density of:
\be(X, Y, G_1, G_2, G_3, G_4)=(f(z), \partial_xf(z), \partial_xf(z), \partial_y f(z), \partial_x^2 f(z), \partial_y\partial x_f(z))\ee
and by $p(x,y;z)$ the joint density of $(f(z), \partial_xf(z)).$
The claim follows from \eqref{eq:KR} using the following chain of equalities:
\begin{align}
k(z)&=\int_{\R^{2\times 2}}|g_1g_4-g_2g_3|P(0,0, g_1, g_2, g_3, g_4;z)dg_1\, dg_2\,dg_3\, dg_4\\
&=\EE\{|G_1G_4-G_2G_3|\quad\textrm{conditioned on}\quad X=Y=0\}\cdot p(0,0;z)\\
&= \EE\{|YG_4-G_2G_3|\quad\textrm{conditioned on}\quad X=Y=0\}\cdot p(0,0;z)\\
&= \EE\{|G_2G_3|\quad\textrm{conditioned on}\quad X=Y=0\}\cdot p(0,0;z)\\
&=\int_{\R^2}|h_1h_2|\rho(0,0, h_1, h_2;z )dh_1\,dh_2.
\end{align}

\end{proof}

\subsection{Unitary invariance}
We endow the extended complex plane $\hat{\C}=\C\cup \{\infty\}$ with the metric induced from the unit sphere under stereographic projection $\sigma:S^2\to \hat{\C}$ (this is called the \emph{Riemann sphere}), defined for $(x,y,t)\in S^2\subset \R^3$ by:
\be\label{eq:stereo}
\sigma:(x, y, t)\mapsto\frac{x+i y}{1-t}.\ee
The action of the orthogonal group on the sphere defines in this way an action of $SO(3)$ on $\hat{\C}$, which we still call the orthogonal action.
\begin{lemma}\label{lemma:invariance}With the above choice of the random polynomials $p$ and $q$, the random rational function $z\mapsto \frac{p(z)}{q(z)}$
is invariant under the orthogonal group. In particular the random lemniscate $\Gamma$ is invariant.
\end{lemma}
\begin{proof}First recall that there is a diffeomorphism $\alpha:\mathbb{C}\textrm{P}^1\to \hat{\C}$ between the complex projective line and the Riemann sphere given by
\be \alpha: [z,w]\mapsto \frac{z}{w}.\ee
Moreover under the map $\alpha$ the metric on the Riemann sphere pulls back to (four times) the Fubini-Study metric.
The special unitary group $SU(2)$ acts by isometries on the projective line with the Fubini-Study metric by:
\be \left(\begin{smallmatrix} \lambda &\mu\\
-\overline\mu&\overline \lambda  \end{smallmatrix}\right)\cdot [z, w]=[\lambda z+\mu w,-\overline{\mu}z+\overline{\lambda}w].\ee
Under the identification provided by $\alpha$, this action is simply the covering homomorphism $SU(2)\to SO(3)$; in particular to prove the invariance of $p/q$ it is enough to prove the invariance under $SU(2)$ of $p/q\circ \alpha$ as a rational function on $\C\textrm{P}^1$. 
The composition $p/q\circ\alpha$ is simply given by:
\be [z, w]\mapsto\frac{~^h\! p(z, w)}{~^h\! q(z,w)}\ee
where $~^h\! p$ and $ ~^h\! q$ are the homogenizations of $p$ and $q$:
\be \label{eq:homogenization} ~^h\! p(z, w)=\sum_{k=0}^na_kz^kw^{n-k}\quad \textrm{and}\quad ~^h\! q(z, w)=\sum_{k=0}^nb_kz^k w^{n-k}.\ee
Finally, the fact that the random polynomials $~^h\! p$ and $~^h\! q$ are $SU(2)$ invariant follows from \cite[Theorem 1, Chapter 12.1]{BCSS}.
\end{proof}
\begin{remark}Notice that the composition of $p$ with a rotation of the Riemann sphere (a special type of M\"obius transformation) can change it to a rational function (and not just simply a polynomial). On the other hand, the composition of both $p$ and $q$ with a rotation changes this rational function $p/q$ to another rational function which (after clearing denominators) is the quotient of two polynomials; the previous lemma indicates that these two polynomials have the same distribution as the original ones.
\end{remark}

\begin{remark}[The Riemann sphere and the complex projective line] The pullback of the round metric $g_{S^2}$ on the sphere via the stereographic projection $\sigma$, defined in \eqref{eq:stereo}, is a Riemannian metric $\sigma^*g_{S^2}$ on $\hat\C$ called the spherical metric. Similarly, one can consider the map $s:\C\textrm{P}^1\to \hat{\C}$ defined by $s([z,w])=z/w$, and consider the pull-back of the Fubini-Study metric $g_{\C\textrm{P}^1}$ (see \cite{GH}). The relation between these two metrics is:
\be \sigma^*g_{S^2}=2 s^*g_{\C\textrm{P}^1}\ee
(in other words $\C\textrm{P}^1$ with the Fubini-Study metric is isometric to the sphere in $\R^3$ of radius $1/2$).
One can consider the lemniscate $\Gamma$ as defined on $\C\textrm{P}^1$ by:
\be\label{eq:gammaproj} \Gamma=\left\{[z,w]\in \C\textrm{P}^1\quad \textrm{s.t.}\quad \left|\frac{~^h\! p(z,w)}{^{h}q(z,w)}\right|=1\right\}.\ee
The two objects \eqref{eq:lemni} and \eqref{eq:gammaproj} are the same, except for the metric (the spherical lemniscate is twice as long as the projective). We notice that adopting the projective viewpoint the unitary invariance of Lemma \ref{lemma:invariance} becomes more clear; moreover the definition of a lemniscate can also be generalized to higher dimensions, simply as the preimage of the unit circle under a rational map $r:\C\textrm{P}^k\to \hat\C.$
\end{remark}





\section{Geometry: the average length of a rational lemniscate}\label{sec:length}
\begin{thm}\label{thm:length}
Let $\Gamma$ be a random lemniscate as defined by \eqref{eq:lemni} and \eqref{eq:pq}.
Then the expectation $\EE|\Gamma|$ of its spherical length is given by:
\be\EE|\Gamma|=\frac{\pi^2}{2}\sqrt{n}.\ee
\end{thm}

\begin{proof}
The proof is divided into two sections.
First we use integral geometry and invariance of the model
to show that $\EE|\Gamma|= 2 \pi^2 C(n),$
where $C(n)$ is a Kac-Rice integral.
Then in Section \ref{sec:Fourier},
we compute $C(n) = \frac{\sqrt{n}}{4}$ using Fourier analysis.

\subsection{Reduction to a Kac-Rice integral}
Below we will let $S^1$ denote the meridian on the sphere $S^2$ that corresponds 
(under stereographic projection) to the real axis in the complex plane.
Identify a point $\theta \in S^1$ on this meridian with the angle measured from the direction $(0,0,-1)$;
the image of $\theta$ under stereographic projection is $\tan \left( \frac{\theta}{2} \right)$.

Applying the integral geometry formula \cite{Howard}, we have:
$$ \frac{|\Gamma|}{\pi} =  \int_{SO(3) }|\Gamma \cap g S^1| dg$$
where the integral over the orthogonal group is with respect to the normalized Haar measure.
Taking the expectation on both sides, we find that
\be\label{eq:IGF}
\EE|\Gamma|= \pi \EE |\Gamma \cap S^1|,
\ee
i.e., the average length is determined by the 
average number of zeros of the function $h(\theta) = f\left( \tan \frac{\theta}{2} \right)$
over the interval $-\pi < \theta < \pi$.

This number can be computed using the Kac-Rice formula.
In terms of the joint probability density $\rho(h,h';\theta)$ of $h(\theta)$ and $h'(\theta)$,
the average number of zeros is given by:
\be
\EE |\Gamma \cap S^1| = \int_{-\pi}^\pi \int_{-\infty}^\infty |h'| \rho(0,h';\theta) dh' d \theta .
\ee
Using again the rotational invariance, the inside integral is independent of $\theta$,
and we have:
\be\label{eq:lengthreduced}
\EE |\Gamma \cap S^1| =  2 \pi \int_{-\infty}^\infty |h'| \rho(0,h';0) dh'.
\ee
Let now $\rho(x_1,x_2)$ denote the joint density of the random variables $X_1,X_2$ defined as:
$$X_1 = h(0) = f(0) = a_0 \ol{a_0} - b_0 \ol{b_0},$$
and 
\begin{align*}
X_2= h'(0) &= \frac{1}{2} \partial_x f (0) \\
&= \frac{1}{2} \left( p'(0)\ol{p(0)} +p(0)\ol{p'(0)} - q'(0)\ol{q(0)} - q(0)\ol{q'(0)} \right)\\
&= \frac{1}{2}\left(a_1 \ol{a_0} + a_0 \ol{a_1} - b_1 \ol{b_0} - b_0 \ol{b_1} \right).
\end{align*}



Referring back to \eqref{eq:IGF} we thus have:
\be\label{eq:length}
\EE|\Gamma| = 2 \pi^2 C(n) ,\ee
where
\be\label{eq:C}
C(n) = \int_{-\infty}^\infty |x_2| \rho(0,x_2) dx_2,\ee
is a Kac-Rice type integral that we compute next.

\subsection{Computation of the Kac-Rice integral}\label{sec:Fourier}
Finally, we compute the integral appearing in \eqref{eq:C} using the method of characteristic functions 
(i.e., Fourier transformation).

By the Fourier inversion formula, we have:
\be\label{eq:inversion}
\rho(x_1,x_2) = \frac{1}{(2\pi)^2} \int \int e^{-i(sx_1 + tx_2)} \hat{\rho}(s,t) ds dt,
\ee
where
\be\label{eq:char}
 \hat{\rho}(s,t) = \EE e^{i(sX_1 + tX_2)}.
\ee

Rearranging the expression in the exponent, we have:
$$sX_1 + tX_2 = s a_0 \ol{a_0} + \frac{t}{2} (a_1 \ol{a_0} + a_0 \ol{a_1}) + s b_0 \ol{b_0} + \frac{t}{2} (b_1 \ol{b_0} + b_0 \ol{b_1}).$$ 

We notice that we can write the expression:
$$q({\bf a})=s a_0 \ol{a_0} + \frac{t}{2} (a_1 \ol{a_0} + a_0 \ol{a_1})$$ 
as a Hermitian quadratic form in the complex normal random vector ${\bf a} = (a_0,a_1)$,
where:
$$q({\bf v})={\bf v}\, \underbrace{\left( \begin{array}{cr}
s & t/2 \\ t/2 & 0 
\end{array}\right)}_{Q} \overline {{\bf v}}^T.$$

Similarly, denoting by ${\bf b} = (b_0,b_1)$, we can write: 
\be s b_0 \ol{b_0} + \frac{t}{2} (b_1 \ol{b_0} + b_0 \ol{b_1})=-q({\bf b}).\ee

By independence, we have:
$$\EE e^{i(sX_1 + tX_2)} = \EE e^{iq({\bf a}) -iq({\bf b})} = \EE e^{iq({\bf a}) } \EE e^{-iq({\bf b})} .$$

Using \cite[Eq. 4(a)]{Turin} while taking $L =  \textrm{diag}(1,n)$, 
$\overline{V} = (0,0)$, and $t=1$ (while treating our variables $s$ and $t$ in $Q$ as parameters), we have:
\begin{align}
 \EE e^{iq({\bf a}) } \EE e^{-iq({\bf b})} &= \frac{1}{\det(\mathbbm{1} - iLQ)\det(\mathbbm{1} + iLQ)} \\
 &= \frac{1}{(1-is+nt^2/4)(1+is+nt^2/4)} \\
 &= \frac{1}{(1+nt^2/4)^2 + s^2}.
\end{align}
Combining this with \eqref{eq:inversion} while setting $x_1=0$, we have:
\begin{align}
\rho(0,x_2) &= \frac{1}{(2\pi)^2} \int_{-\infty}^\infty \int_{-\infty}^\infty \frac{e^{-itx_2}}{(1+n t^2/4)^2 + s^2} ds dt \\
&= \frac{1}{4\pi} \int_{-\infty}^\infty \frac{e^{-itx_2}}{(1+nt^2/4)} dt \\
&= \frac{1}{2 \sqrt{n}} e^{-2|x_2|/\sqrt{n}} ,\\
\end{align}
where we have performed the above integrations using residues.
Finally, the integral \eqref{eq:C}
is now elementary to compute:
\begin{align*}
C(n) = \int_{-\infty}^\infty |x_2| \rho(0,x_2) dx_2 &= \frac{1}{2\sqrt{n}}\int_{-\infty}^\infty |x_2| e^{-2|x_2|/\sqrt{n}} dx_2 \\
&= \frac{\sqrt{n}}{8}\int_{-\infty}^\infty \tau e^{-\tau} d\tau \\
&= \frac{\sqrt{n}}{4}\int_{0}^\infty \tau e^{-\tau} d\tau = \frac{\sqrt{n}}{4},\\
\end{align*}
and \eqref{eq:length} becomes $ \displaystyle \EE|\Gamma| = \frac{\pi^2}{2} \sqrt{n} ,$ as desired.
\end{proof}
\section{More Geometry: the average number of meridian tangents}\label{sec:meridian}
Fixing a point $z\in \hat{\C}$, say the north pole, consider the projection on the equator perpendicular to $z$:
\be h:\hat\C\backslash\{z, -1/\overline z\}\to S^1.\ee
If $(\theta, \phi)$ are the standard spherical coordinates, then we simply have:
\be\label{eq:h} h(\theta, \phi)=\theta.\ee
For every measurable subset $A\subset \hat{\C}$, we denote by $\nu(A)$ the number of critical points in $A$ of the restriction $h|_\Gamma$ (notice that with probability one $\Gamma$ doesn't pass through $z$ or its antipodal, 
where $h$ is not defined). The number $\nu$ measures the number of times $\Gamma$ is tangent to a meridian, naively the number of time $\Gamma$ ``changes direction''. Because of invariance under the orthogonal group, the expectation of $\nu$ doesn't depend on the point $z$.
\begin{thm}\label{thm:tangents}Let $A\subset \hat{\C}$ be a measurable set. Then:
\be \EE\nu(A)\sim |A|\frac{n(32-\sqrt{2})}{112 \pi}\ee
(here $|A|$ denotes the spherical measure of $A$).
In particular $\EE \nu(\hat{\C})\sim \frac{n(32-\sqrt{2})}{28}.$
\end{thm}
\begin{proof}
As in the proof of Theorem \ref{thm:length}, the proof is divided into two parts.
First we use invariance of the model
to reduce to a Kac-Rice integral, which we then compute in Section \ref{sec:Fouriermeridian}, using Fourier analysis.
\subsection{Reduction to a Kac-Rice integral}
Let $S^2$ be the unit sphere in $\R^3\sim \C\times \R$ and let us consider on it spherical coordinates $(\theta, \phi)$ centered at the point $(0,1, 0)$ (notice that these \emph{are not} the standard spherical coordinates, but with this choice the computations below simplify). Let also $\sigma:S^2\backslash \{(0,0,1)\} \to \R^2$ be the stereographic projection:
\be \label{eq:polar}\sigma(\theta, \phi)=(x(\theta, \phi), y(\theta, \phi)).\ee
Let $g(z)=g(x+iy)=|p(x+iy)/q(x+iy)|^2-1$ and set:
\be \tilde{g}(\theta, \phi)=g(\sigma(\theta, \phi))\ee
Then $\nu(A)$ equals the number of solutions in the region $R=\{(\theta, \phi)\,|\, \sigma(\theta, \phi)\in A\}$ of the system :
\be \left\{\begin{matrix}\tilde{g}=0\\ \partial_\phi \tilde{g}=0\end{matrix}\right.\ee
In particular, using the Kac-Rice formula as in Lemma \ref{lemma:condi}:
\be\label{eq:KRstart} \EE\nu(A)=\int_{R}k(\theta, \phi)d\phi d\theta ,\ee 
where $k$ is the Kac-Rice density:
\be k(\theta, \phi)=\int_{\R^2}|h_1h_2|\tilde\rho(0,0, h_1, h_2, \theta, \phi)dh_1 dh_2\ee
and $\tilde\rho$ is the joint density of the random vector $(\tilde g, \partial_\phi \tilde g, \partial_\theta \tilde g, \partial^2_{\theta}\tilde g)$ evaluated at $(\theta, \phi)$.
Using the invariance under rotation, we see that $\partial_\theta k=0$ and $k=k(\phi)$; moreover composing $\tilde g$ with a rotation of an angle $-\phi+\pi/2$ around the $y$-axis, we see that (as random variables):
\be (\tilde g, \partial_\phi \tilde g, \partial_\theta \tilde g, \partial^2_{\theta}\tilde g)(\theta, \phi)\sim (\tilde g, \partial_\phi \tilde g,\sin \phi \,\partial_\theta \tilde g, \partial^2_{\theta}\tilde g)(0, \pi/2).\ee
In particular, still denoting by $\tilde\rho$ the joint density of $(\tilde g, \partial_\phi \tilde g,\partial_\theta \tilde g, \partial^2_{\theta}\tilde g)(0, \pi/2)$:
\be\label{eq:K-R} k(\phi)=\sin\phi\cdot \int_{\R^2}|h_1h_2|\tilde\rho(0,0, h_1, h_2)dh_1 dh_2.\ee
Recalling our definition of $\tilde g$, we have:
\begin{align}\label{eq:coord1} 
\tilde{g}(0, \pi/2)=g(0),\quad \partial_\phi\tilde g(0,\pi/2)=\partial_xg(0)\partial_\phi x(0,\pi/2)\quad \partial_\theta\tilde g(0,\pi/2)=\partial_yg(0)\partial_\theta x(0,\pi/2)
\end{align}
\begin{align}\label{eq:coord2}
(\partial^2_\phi \tilde g)(0,\pi/2)=\partial^2_xg(0)(\partial_\phi x(0,\pi/2))^2+\partial_xg(0)\partial^2_{\phi}x(0,\pi/2).
\end{align}
In the chosen coordinates we have:
\be x(0,\phi)=\tan\left(\frac{\phi}{2}-\frac{\pi}{4}\right)\quad \textrm{and}\quad y\left(\theta, \frac{\pi}{2}\right)=\tan \left(\frac{\theta}{2}\right),\ee
which substituted into \eqref{eq:coord1} and \eqref{eq:coord2} gives:
\begin{align} \partial_\phi\tilde g(0,\pi/2)=\frac{1}{2}\partial_xg(0)\quad \partial_\theta\tilde g(0,\pi/2)=\frac{1}{2}\partial_yg(0), \quad(\partial^2_\phi \tilde g)(0,\pi/2)=\frac{1}{4}\partial^2_xg(0). \end{align}
Using these equations into \eqref{eq:K-R}, we can write:
\be\label{eq:k1} k(\phi)=\frac{1}{4}\sin\phi\cdot \int_{\R^2}|h_1h_2|\rho(0,0, h_1, h_2)dh_1 dh_2\ee
where now $\rho$ denotes the joint density of $(g(0), \partial_xg(0), \partial_yg(0),\partial^2_{x}g(0)).$

\subsection{Computation of the Kac-Rice integral}\label{sec:Fouriermeridian}
Notice now that the integral in \eqref{eq:k1} equals the Kac-Rice density at zero for the system of random equations:
\be 
\left\{\begin{matrix}g(x,y)=0\\ \partial_x g(x,y)=0\end{matrix}\right.
\ee
Denoting by $f=|p|^2-|q|^2$, with probability one, the above system has the same solutions of the random system:
\be\label{eq:system1} \left\{\begin{matrix}f=0\\\partial_xf=0\end{matrix}\right.\ee
Hence their Kac-Rice densities at zero coincide, and:
\be\label{eq:KRfinal} (*):= \int_{\R^2}|h_1h_2|\rho(0,0, h_1, h_2)dh_1 dh_2=\int_{\R^2}|h_1h_2|\tilde\rho(0,0, h_1, h_2)dh_1 dh_2 ,\ee
where $\tilde \rho$ is the joint density of the random variables 
$$(X_1,X_2, H_1, H_2)=(f(0),\partial_xf(0), \partial_yf(0), \partial^2_{x}f(0)):$$
\be f(0)=a_0\overline a_0-b_0\overline b_0, \quad \partial_xf(0)=a_1\overline a_0+a_0\overline a_1-b_1\overline b_0-b_0\overline b_1, \quad \partial_yf(0)=i(a_1\overline a_0-a_0\overline a_1-b_1\overline b_0+b_0\overline b_1)\ee
\be \partial_x^2f(0)=a_2\overline a_0+2a_1\overline a_1+a_2\overline a_0-b_2\overline b_0-2b_1\overline b_1-b_0\overline b_2\ee
(we have used the identities $\partial_x=\partial_z+\partial_{\overline{z}}$ and $\partial_y=i(\partial_z-\partial_{\overline z})$ to compute the derivatives of $f$).

We proceed now in a similar way as in the proof of Theorem \ref{thm:length}, keeping the notation close to it in order to stress analogies. Denoting by ${\bf a}=(a_0, a_1, a_2)$ and ${\bf b}=(b_0, b_1, b_2)$, we can write:
\be sX_1+tX_2+uH_1+wH_2=q({\bf a})-q({\bf b})\ee
where now $q$ is the Hermitian form:
\be\label{eq:Q3} q({\bf v})={\bf v}\underbrace{\left(\begin{array}{ccc}s & t-iu & w \\t+ iu & 2w & 0 \\w & 0 & 0\end{array}\right)}_{Q} {\overline{\bf v}}^T.\ee
Using the Fourier inversion formula, we can write:
\begin{align}
(*)&=\int_{\R^2}|h_1h_2|\frac{1}{(2\pi)^4}\int_{\R^4}e^{-i(h_1u+h_2 w)}\EE e^{i q({\bf a})-iq({\bf b})}ds\, dt\, du\, dw\,dh_1\,dh_2\\
&=\frac{1}{(2\pi)^4}\int_{\R^2}|h_1h_2|\int_{\R^4}e^{-i(h_1u+h_2 w)}\EE e^{i q({\bf a})}\EE e^{-iq({\bf b})}ds\, dt\, du\, dw\,dh_1\,dh_2\\
\label{eq:explicit}&=\frac{1}{(2\pi)^4}\int_{\R^2}|h_1h_2| \int_{\R^4}e^{-i(h_1u+h_2 w)}F^+_n(s,t,u, w)F_n^-(s,t,u,w)ds\, dt\, du\, dw\,dh_1\,dh_2.
\end{align}
For the explicit computation of $F_n^{\pm}$ we can again use 
\cite[Eq. 4(a)]{Turin} in this case with $Q$ given by \eqref{eq:Q3} and
$L = \textrm{diag}(1, n, {n \choose 2})$:
\be F_n^{\pm}(s,t,u,w)=\frac{1}{1+nt^2+nu^2-2nsw-nw^2/2-n^2w^2/2\pm i(n^2w^3-n^3w^3-2nw-s)}.\ee
Notice that we can rewrite $F_n^+F_n^-$ as:
\be F_n^+F_n^-=\frac{1}{ \left(1+n t^2+n u^2-2 n s w-n w^2/2+n^2 w^2/2\right)^2+\left( n^2 w^3- n^3 w^3-2n w-s\right)^2}.\ee
Performing the $s$ integration using residues, we obtain:
\be (*)= \frac{1}{(2\pi)^4} \int_{\R^2}|h_1h_2|\int_{\R^3}e^{-i(h_1u+h_2 w)} F_{2, n}(t,u, w) dt\, du\, dw\,dh_1\,dh_2\ee
where:
\be F_{2,n}(t, u,w)=\frac{2\pi}{ 2+2 n \left(t^2+u^2\right)+n (9 n-1) w^2+4 (n-1) n^3 w^4}.\ee
Performing as well the $t$ integration using residues:
\be (*)= \frac{1}{(2\pi)^4} \int_{\R^2}|h_1h_2| \frac{\pi^2}{n^2}\int_{\R^2}e^{-i(h_1u+h_2 w)} F_{3, n}(u, w)  du\, dw\,dh_1\,dh_2\ee
where now:
\be F_{3,n}(u, w)=\frac{1}{\sqrt{u^2+1/n + (9 n-1) w^2/2+2 (n-1) n^2 w^4}}.\ee
Since $F_{3, n}(u, w)$ is an even function of $u$ we have:
 \begin{align}(*)&= \frac{1}{16 \pi^2 n^2} \int_{\R^2}|h_1h_2|\int_{\R^2}e^{-i(h_1u+h_2 w)} F_{3, n}(u, w)  du\, dw\,dh_1\,dh_2\\
 &= \frac{1}{8 \pi^2 n^2} \int_{\R^2}|h_1 h_2|  \int_{\R}e^{-ih_2 w} \int_{0}^\infty \cos(uh_1) F_{3, n}(u, w)  du\, dw\,dh_1\,dh_2.\end{align}
 For the $u$ integration, we use the identity \cite[Sec. 3.754, \# 2.]{Grad}:
 $$ \int_0^\infty \frac{\cos (ax)}{\sqrt{\beta^2+x^2}} dx = K_0(a \beta) ,$$
 where $K_0$ is the zeroth-order modified Bessel function of the second kind. 
 This results in:
 \be (*)=  \frac{1}{8 \pi^2 n^2} \int_{\R^2}|h_1 h_2| \int_{\R}e^{-ih_2 w}K_0\left(\frac{|h_1|}{F_{3,n}(0, w)}\right)dw\,dh_1\,dh_2.\ee
Using Fubini's theorem and noticing that the integrand is an even function of $h_1$, we have:
  \begin{align}(*)&=   \frac{1}{8 \pi^2 n^2}\int_{\R^2}|h_2| e^{-ih_2 w} \int_{\R} |h_1| K_0\left(\frac{|h_1|}{F_{3,n}(0, w)}\right)dh_1 \,dw \,dh_2\\
 &=   \frac{1}{4 \pi^2 n^2}\int_{\R^2} |h_2| e^{-ih_2 w} \int_{0}^{\infty} h_1 K_0\left(\frac{h_1}{F_{3,n}(0, w)}\right)dh_1 \,dw \,dh_2.\end{align}
 Using \cite[Sec. 6.561, \#16.]{Grad} to evaluate the inside integral, we have:
 \be (*)= \frac{1}{4 \pi^2} \int_\R|h_2| \int_\R e^{-ih_2 w}F_{4,n}(w)dw dh_2, \ee
 where:
 \be F_{4,n}(w)=\frac{2 }{2+n w^2 \left(-1+n \left(9+4 (-1+n) n w^2\right)\right)}.\ee
 The change of variables $v=n w$, $h_2=\tau n$ gives:
 \begin{align} (*)&= \frac{1}{4 \pi^2} \int_\R|h_2| \int_\R e^{-ih_2 \frac{v}{n}}F_{4,n}(v/n)\frac{dv}{n} \,dh_2 \\
 &= \frac{n}{4 \pi^2}\int_\R|\tau| \int_\R e^{-i\tau v}F_{4,n}(v/n)dv\,d\tau.\end{align} 
 Applying Lebesgue's dominated convergence theorem to the sequence $F_{4,n}(v/n)$, we see that:
 \be (*)\sim \frac{n}{4 \pi^2} \int_\R|\tau|\int_{\R}e^{-iv \tau}\frac{2}{2+9v^2+4v^4}dv\,dw.\ee
 Since the integrand is even in $v$, we have:
 \begin{align} (*)&\sim \frac{n}{4 \pi^2} \int_\R|\tau| \int_{\R}\cos(\tau v)\frac{2}{2+9v^2+4v^4}dv\,d\tau\\
 &= \frac{n}{4 \pi^2}  \int_\R|\tau|\int_{0}^\infty \frac{\cos(\tau v)}{(v^2+2)(v^2+1/4)}dv\,d\tau \\
(\text{by \cite[Sec. 3.728, \#1.]{Grad}}) \quad &= \frac{n}{4 \pi^2} \int_\R|\tau | \frac{ \pi}{7} \left(4 e^{-\frac{|\tau |}{2}}-\sqrt{2} e^{-\sqrt{2} |\tau|}\right) d\tau\\
   &=\frac{n}{14\pi} \int_0^\infty \tau\left(4 e^{-\frac{\tau}{2}}-\sqrt{2} e^{-\sqrt{2} \tau}\right) d\tau\\
 \label{eq:asymp}&=n \frac{32-\sqrt{2}}{28 \pi}.
 \end{align}
\subsection{End of the proof}We can now substitute the asymptotic \eqref{eq:asymp} into  \eqref{eq:k1} and obtain:
\be k(\phi)\sim n \left( \frac{32-\sqrt{2}}{112 \pi}\right)\sin\phi\ee
which substituted in \eqref{eq:KRstart} finally gives:
\begin{align}  \EE\nu(A)&\sim n \left( \frac{32-\sqrt{2}}{112 \pi}\right)\int_{R}\sin \phi \,d\phi\, d\theta, \quad R=\{(\theta, \phi)\,|\, \sigma(\theta, \phi)\in A\} \\
&= n \left( \frac{32-\sqrt{2}}{112 \pi}\right)|A|.\end{align}
\end{proof}

\section{Topology}\label{sec:topology}

\subsection{Arrangements and the nesting graph of a curve on the sphere}
Given two manifold pairs $(A_1, B_1)$ and $(A_2, B_2)$, where each $A_i$ is a submanifold of $B_i$, we say that they are isotopic if there exists a homeomorphism $\psi:B_1\to B_2$ such that $\psi$ restricts to a homeomorphism $\psi|_{A_1}:A_1\to A_2;$ in this case we write:
\be (B_1, A_1)\sim (B_2, A_2).\ee
 In the case each $A_i$ is a real curve (possibly with many components) in a region $B_i\subset \R^2$, 
 the notion of isotopy for manifold pairs is stronger than just $B_1$ being homeomorphic to $B_2$ and $A_1$ to $ A_2$, because it depends on the way each $A_i$ is embedded in $B_i$. The isotopy class of a manifold pair $(A, B)$ is usually called an \emph{arrangement} of $A$ in $B$.
 
In the case of a curve $C$ on the sphere $S^2$, the arrangement $(S^2, C)$ is captured by the \emph{nesting graph} $T(S^2, C)$, 
built in the following way. Vertices of the graph are the components of $S^2\backslash C$
(by Alexander's duality \cite[Theorem 3.44]{Hatcher} the number of vertices is $b_0(C)+1$), and one edge runs between two components if and only if they have a (unique) common boundary. 
Since every component of $C$ separates $S^2$ into two open contractible sets, the graph $T(S^2, C)$ is indeed a tree. 

In the case of a curve $C$ in the plane $\R^2$ the nesting tree becomes a \emph{rooted} tree (the root is the unique unbounded component of $\R^2\backslash C$, see figure \ref{fig:tree}). Stereographic projection $\sigma:S^2\backslash\{p\}\to \R^2$ from one point $p\notin C$ induces an isomorphism of graphs $T(S^2, C)\simeq T( \R^2, \sigma(C))$ (the root of $T(S^2, C)$ is the component containing $p$).

It is not difficult to prove (and we will use this fact in the sequel) that the isomorphism class of trees (respectively of rooted trees) characterizes the isotopy class $(S^2, C)$ (respectively $(\R^2, C)$).
\begin{figure}
\begin{center}

\includegraphics[width=0.7\textwidth]{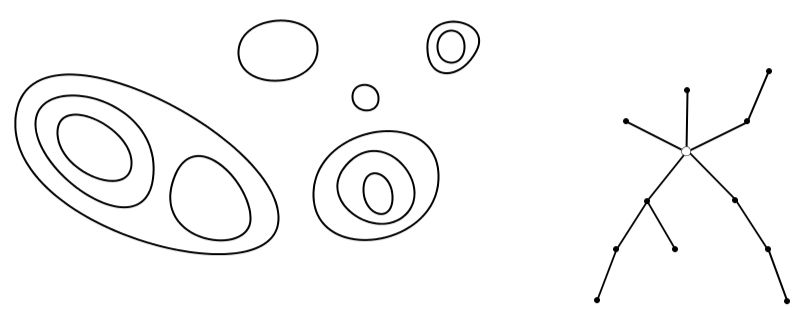}
\caption{An arrangement of curves in $\R^2$ and the associated rooted tree.}
\label{fig:tree}
\end{center}
\end{figure}
\subsection{Arrangements of rational lemniscates}\label{sec:H16} Notice that the lemniscate $\{|p/q|=1\}$ can also be seen as the preimage of the unit circle $S^{1}\subset \hat{\C}$ under the holomorphic map $p/q=r:\hat{\C}\to \hat{\C}$:
\be\Gamma=\left\{\left|\frac{p(z)}{q(z)}\right|=1\right\}=r^{-1}(S^1).\ee
We will say that the lemniscate is \emph{nondegenerate} if the map $r$ is transversal to $S^1$ (in other words, zero is a regular value of $z\mapsto p(z)\overline{p(z)}-q(z)\overline{q(z)}$).
This implies that $\Gamma$ consists of smooth components.
The following upper bound on the number of components is well-known.
\begin{prop}\label{prop:components}
Let $\Gamma\subset \hat\C$ be a nondegenerate rational lemniscate of degree $n$. 
Then $\Gamma$ has at most $n$ components.
\end{prop}
\begin{proof}
The lemniscate $\Gamma$ is the level set of the logarithmic potential 
$\psi(z) := \log |p(z)| - \log |q(z)|$.
By the maximum principle, each component of $\Omega :=\{ z \in \hat \C : \psi(z) < 0 \}$ contains a zero of $p(z)$,
since $\psi$ is a non-constant function vanishing on the boundary
and harmonic except at the zeros of $p$.
Each component of $\Gamma$ is a boundary component of some component of $\Omega$.
Thus, the number of components of $\Gamma$ is bounded by the number of components of $\Omega$.

Alternative proof: The rational function $r:\hat\C\to \hat\C$ defining $\Gamma=r^{-1}(S^1)$ is a degree-$n$ branched covering of the sphere, 
and hence for all $w\in \hat\C$ except finitely many points (the critical values of $r$) we have $ \#r^{-1}(w)=n$.


Since $r$ is transversal to $S^1$, if $C_1, \ldots, C_m$ are the components of $\Gamma$, then for every $k=1, \ldots, m$
the restriction:
\be r|_{C_k}:C_k\to S^1\ee
is a degree $n_k$ covering of the circle.  
Thus, if $w\in S^1$ we have:
\be \#r^{-1}(w)=|n_1|+\cdots+|n_m|=n,\ee
which implies $m\leq n.$
\end{proof} 

We will show that any arrangement of $n$ many circles in $\hat\C$ is isotopic to $(\Gamma, \hat{\C})$ for some rational lemniscate $\Gamma$ of degree $n$.
This type of result---that ``what is possible topologically is possible rationally'' has been observed in related problems,
such as the inverse image of the real line under a rational function with real coefficients \cite{BorceaShapiro}
(these results also resemble one for harmonic polynomial zero sets proved in \cite[Thm. 1.3]{EremJakNad}).
The present case is certainly known among experts, 
but we were unable to find a reference specifically adapted to our purposes,
so we provide a proof here.
We shall need the following lemma.

\begin{lemma}\label{lemma:build}
Let $\Gamma\subset \hat\C$ be a nondegenerate rational lemniscate of degree $n$ and 
$C\sim S^1\subset \hat{\C}$ be a  circle
such that $\Gamma$ is contained in the region exterior to $C$. 
Then there exists a nondegenerate rational lemniscate $\Gamma'$ of degree $n+1$ such that:
\be\label{eq:ISO} (\hat\C, \Gamma')\sim (\hat\C, \Gamma \cup C).\ee
\end{lemma}
\begin{proof}Consider thus $p,q$ such that $r=p/q:\hat\C\to \hat\C$ is transversal to $S^1$ and:
\be\Gamma=r^{-1}(S^1).\ee
Let $z_0$ be a point interior to $C$.
It suffices to introduce a small component around $z_0$ disjoint from $\Gamma$.
We will accomplish this by perturbing $r$ with a simple pole.
By possibly composing with a rotation we may assume that $z_0=0$,
and by possibly exchanging the role of $p$ and $q$ we may also assume that $|p(0)/q(0)|<1$. 
Consider the family of lemniscates:
\be \Gamma_\e=\left\{\left|\frac{p(z)}{q(z)}+\frac{\e}{z}\right|=1\right\}, \quad \e \geq 0 .\ee
Let also $D=D(0, \delta)$ be a small enough closed disk such that $D\cap \Gamma=\emptyset$  (in particular $\left| r(z) \right|<1$ for $|z| = \delta$). 
We claim that for $\e>0$ small enough, in the complement of $D$ the lemniscate $\Gamma_\e$ has the same arrangement as $\Gamma$: 
\be (\hat{\C}\backslash D, \Gamma_\e\cap \hat{\C}\backslash D )\sim (  \hat{\C}\backslash D, \Gamma\cap  \hat{\C}\backslash D).\ee 
Notice that $\hat{\C}\backslash D$ is itself a disk in the chart at infinity; 
on this chart the rational function $r_\e$ defining $\Gamma_\e=r_\e^{-1}(S^1)$ is given by:
\be\label{eq:repsilon} r_\e(w)=\frac{~^h\!p(1, w)}{~^h\! q(1, w)}+\e w\ee
(the definition of $~^h\!p$ and $~^h\!q$ was given in \eqref{eq:homogenization}).

It follows from \eqref{eq:repsilon} that on the disk $\hat{D}=\{|w|\leq 1/\delta\}$ the $C^1$-norm of the difference $r_\e-r$ is a continuous function of $\e$ and it is zero for $\e=0$. 
Consequently $\|r_\e-r\|_{C^1(\hat D, \C)}$ can be made arbitrarily small and, since $r$ is transversal to $S^1$, for $\e>0$ small enough $r_\e$ is also transversal to it. 
Moreover for such $\e$ the maps $r|_{\hat{D}}$ and $r_\e|_{\hat D}$ are homotopic, and by Thom's Isotopy Lemma \cite{mat} (see also \cite[Section 1.5]{GorMacph}):
\be\label{eq:iso3} ( \hat{D}, \Gamma_\e\cap \hat{D})\sim (\hat{D}, \Gamma\cap \hat D).\ee
Finally in the disk $D$ there must be exactly one component of $\Gamma_\e$. 

Such component must exist because $|r_\e(0)|=\infty$ and for $\e>0$ small enough we also have $|r_\e(z)|<1$ for all $|z|=\delta$.
Moreover, since $\Gamma_\e$ has degree $n+1$, it has at most $n+1$ components (by Proposition \ref{prop:components}), $n$ of which are already in $\hat{D}$. 
Hence, $(D, \Gamma_\e\cap D)\sim (D, S^1)$, which together with \eqref{eq:iso3} implies \eqref{eq:ISO} (the trees $T(\hat\C, \Gamma_\e)$ and $T(\hat\C, \Gamma\cup C)$ are isomorphic).

For the nondegeneracy of $\Gamma_\e$, we have already chosen $\e$ small enough so that $r_\e|_{\hat{D}}$ is transversal to $S^1$, 
and it remains to show that we can take $\e>0$ (possibly even smaller) in order to have $r_\e|_D$ also transversal to $S^1.$ 
It suffices to show that the holomorphic derivative $r_\e'(z)$ does not vanish on the component of $\Gamma_\e$ contained in $D$.
Since $r$ is holomorphic in a neighborhood of $D$, then $|r'|<c_1$ on $D$, for some $c_1 < \infty$. 
By assumption $|r(z)|<c_0<1$ on the disk $D$,
and in order for $z$ to solve $|r(z)+\e/z|=1$ we must have $|z|<\frac{\e}{1-c_0}$.
Indeed, this follows from:
\be 1=\left|r(z)+\frac{\e}{z}\right|< |r(z)|+\frac{\e}{|z|} < c_0+\frac{\e}{|z|}.\ee
Since $c_0, c_1$ are fixed, we can choose $\e>0$ small enough so that:
\be \frac{\e}{1-c_0}<\left(\frac{\e}{c_1}\right)^{1/2}.\ee
Thus, on the component of $\Gamma_\e$ in the disk $D$ we have:
$$|z|^2 < \left(\frac{\e}{1-c_0}\right)^2 < \frac{\e}{c_1},$$
and in particular $\frac{\e}{|z|^2} > c_1$. This implies that $r_\e'(z) \neq 0$, since
\be |r_\e'(z)| = \left| r'(z)-\frac{\e}{z^2} \right| \geq \frac{\e}{|z|^2} - |r'(z)| \geq \frac{\e}{|z|^2} - c_1 .\ee
Hence $r_\e$ is transversal to $S^1$.

\end{proof}
\begin{prop}\label{prop:H16}
Let $A\subset \hat{\C}$ be  a curve consisting of 
a disjoint union of $n$ many circles. 
Then there exists a nondegenerate lemniscate $\Gamma$ of degree $n$ such that:
\be (\hat\C, \Gamma)\sim (\hat{\C}, A).\ee
\end{prop}
\begin{proof}
The proof is by induction, and Lemma \ref{lemma:build} provides the inductive step.
Let $\Lambda \subset \C$ be a nondegenerate rational lemniscate with rooted tree $T=T(\C,\Lambda)$ and consider the tree $T'$ obtained by adding one edge (and consequently one vertex) to a leaf of $T$; 
then Lemma \ref{lemma:build} guarantees that there exists  a nondegenerate rational lemniscate $\Lambda'\subset \C$ whose rooted tree $T(\C, \Lambda')$ is isomorphic to $T'$. 

Now in order to construct $\Gamma$ using the above inductive step 
consider the stereographic projection $\sigma:\hat\C\backslash\{p\}\to \C$ from a point $p$ not on $A$ and denote by $T$ the rooted tree of $(\C, \sigma(A)).$  
Since $T$ can be built starting from the root adding one leaf at a time, the result then follows. \end{proof}
 
 The statement in Proposition \ref{prop:H16} 
 and the inductive procedure on a nesting tree resembles 
 the topological classification of so-called "algebraic droplets"
 (another class of real-algebraic curves with a special connection to potential theory)
 provided recently by S-Y. Lee and N. Makarov \cite{LeeMak}.

 \subsection{Local arrangement of a random lemniscate}\label{sec:local}
  We introduce the following notation: given a closed disk $D\subset \R^2$ and a $C^1$ map $f:D\to \C$ we define:
\be \|f\|_{C^1(D, \C)}=\sup_{z\in D}|f(z)|+\sup_{z\in D}\|Jf(z)\|\ee
where as a norm for the Jacobian matrix $Jf(z)$ we take (all norms in $\R^{2\times 2}$ are equivalent):
\be \|Jf(z)\|=\left\| \left(\begin{array}{cc}a & b \\c & d\end{array}\right)\right\|=(a^2+b^2+c^2+d^2)^{1/2}.\ee
Notice that the natural inclusion $C^1(D, \R)\hookrightarrow C^1(D, \C)$ is isometric. Moreover if $D\subset \C$ and $f$ is holomorphic, then writing $z=x+iy$ and $f(z)=(u(z), v(z))$ we have:
\be Jf(z)=\left(\begin{array}{cc}u_x(z) & u_y(z) \\-u_y(z) & u_x(z)\end{array}\right)\quad \textrm{and}\quad f'(z)=u_x(z)+iv_x(z).\ee
and in particular:
\be\label{eq:holder} \|Jf(z)\|=\sqrt{2}|f'(z)|.\ee
The resulting topology on $C^1(D, \C)$ is the \emph{Withney topology} (recall that since $D$ is compact the weak and the strong topologies on $C^1(D, \C)$ in the sense of \cite[Chapter 2.1]{Hirsch} coincide).
\begin{thm}\label{thm:isotopy}
Let $A\subset \R^2$ be a curve consisting of finitely man circles and fix $\rho>0$.
Given $z\in S^2$ let $D_n(z)$ denote the open disk $\textrm{\emph{int}}(D_{S^2}(z, \rho n^{-1/2})).$ 
There exists a constant $a>0$ (independent of both $z$ and $n$) such that
for a random lemniscate $\Gamma$ of degree $n$ we have
\be \PP\left\{(D_n(z),\Gamma \cap D_n(z))\sim (\R^2, A)\right\}>a.\ee
\end{thm}
\begin{proof}First we notice that, by invariance of the model, it is enough to prove the statement for the special case $z=0$ (thinking of the Riemann sphere as $S^2\simeq \hat{\C}$). 

By Proposition \ref{prop:H16}, there exist polynomials $\alpha, \beta\in \C[z]$ of degree $d=b_0(A)$ such that:
\be \left(\textrm{int}(D(0, R)), \left\{\left|\frac{\alpha}{\beta}\right|=1\right\}\right)\sim  (\R^2, A)\ee
for some $R>0$ (here $D(0,R)$ denotes the closed disk $\{|z|\leq R\}$). 
By possibly scaling $\alpha/\beta$ by a constant we may assume that $R=1$; 
moreover Proposition \ref{prop:H16} also guarantees that the equation $|\frac{\alpha}{\beta}|=1$ is regular, i.e. zero is a regular value of the function $\varphi=|\alpha|^2-|\beta|^2=\alpha\bar \alpha-\beta\bar \beta$.

We will need the following lemma; in the lemma and throughout the rest of the proof we will let $D$ denote the \emph{closed} unit disk $D(0, 1)$.

\begin{lemma}\label{lemma:stable}
For $\alpha, \beta \in \C[z]_d$ as above 
there exist neighborhoods $U_\alpha\subset \C[z]_d$ of $\alpha $ and $U_\beta\subset \C[z]_d$ of $\beta$ and $\delta>0$ such that for every $\tilde\alpha\in U_\alpha$ and $\tilde\beta\in U_\beta$ and for every pair of polynomials $g_1, g_2$ whose $C^1(D, \C)$-norms 
are bounded by $\delta$ we have:
\be \label{eq:iso1}\left(D, \left\{\left|\frac{\tilde\alpha+g_1}{\tilde\beta+g_2}\right|=1\right\}\right)\sim \left(D, \left\{\left|\frac{\alpha}{\beta}\right|=1\right\}\right).\ee
\end{lemma}
\begin{proof}
First recall that, by the Transversality Theorem \cite[Theorem 2.1 (b)]{Hirsch}, maps defined on a compact manifold (the unit disk $D$) transversal to a compact submanifold (the origin in $\R$) form an open and dense set in the $C^1(D, \R)$-topology. 
Moreover, Thom's Isotopy Lemma again ensures that if two such maps $f_1, f_2$ are close enough in the $C^1(D, \R)$ topology, their zero sets are ambient isotopic:
\be(D, \{f_1=0\})\sim (D, \{f_2=0\}).\ee 
Thus there exists $\delta_1>0$ such that for every function $\varphi_\e:D\to \R$ with $\|\varphi-\varphi_\e\|_{C^1(D, \R)}\leq \e$ we have $(D, \{\varphi=0\})\sim (D, \{\varphi_\e=0\})$. Consider the map:
\be \eta:\C[z]_d\times \C[z]_d\times C^1(D, \C)\times C^1(D, \C)\to \R\ee
defined by:
\be\eta(\tilde \alpha, \tilde\beta, g_1, g_2)=\||\alpha|^2-|\beta|^2-|\tilde\alpha+g_1|^2+|\tilde\beta+g_2|^2\|_{C^1(D, \R)}.\ee
The map $\eta$ is continuous (it is the composition of continuous functions) and $\eta(\alpha, \beta, 0, 0)=0$. 
Thus there exists $\delta>0$ such that the open set $\eta^{-1}(0, \e)$ 
contains an open set  $W\subset \C[z]_d\times \C[z]_d\times C^1(D, \C)\times C^1(D, \C)$, with $(\alpha, \beta, 0,0)\in W$, of the form:
\be W= U_\alpha\times U_\beta\times B_{C^1(D, \C)}(0, \delta)\times  B_{C^1(D, \C)}(0, \delta)\ee
and this proves the claim.\end{proof}

Let now $f(z)=|p(z)|^2-|q(z)|^2$ and consider the rescaling:
\be F(z)=f( \rho n^{-1/2}z).\ee
Notice now that the statement of the theorem for $f|_{D_n(0)}$ is equivalent to the same statement for $F|_{\textrm{int}(D)}$, i.e. it suffices to prove that for some $a>0$ (independent of $n$) we have:
\be\label{eq:iso2} \PP\left\{(\textrm{int}(D), \{f=0\} \cap \textrm{int}(D))\sim (\R^2, A)\right\}>a.\ee
We will write the (rescaled) random polynomial $p$ in the form:
\begin{align}\label{eq:pdec} p(\rho n^{-1/2}z)&=\sum_{k=0}^da_kn^{-k/2}\rho^kz^k+\sum_{k=d+1}^{\ell}a_kn^{-k/2}\rho^kz^k+\sum_{k=\ell+1}^na_kn^{-k/2}\rho^kz^k\\
&=P_1(z)+P_2(z)+P_3(z)\end{align} and similarly for $q$:
\be \label{eq:qdec}q(\rho n^{-1/2}z)=Q_1(z)+Q_2(z)+Q_3(z)\ee
(the choice of the integer $\ell$ will be determined below using Lemma \ref{lemmaL}).

By Lemma \ref{lemma:stable}, if the six events
\begin{align} E_{1,1,}=\{P_1\in U_\alpha\}, \quad E_{1, 2}=\{P_2\in B_{C^1(D, \C)}(0, \delta/2)\}, \quad E_{1, 3}=\{P_3\in B_{C^1(D, \C)}(0, \delta/2)\},\\
E_{2,1}=\{Q_1\in U_\beta\}, \quad E_{2, 2} = \{Q_2\in B_{C^1(D, \C)}(0, \delta/2)\}, \quad E_{2, 3}=\{Q_3\in B_{C^1(D, \C)}(0, \delta/2)\},\end{align}
all occur, then we have the desired isotopy:
\be (\textrm{int}(D), \{f=0\} \cap \textrm{int}(D))\sim (\R^2, A).\ee
Thus, it suffices to bound from below (by a positive constant independent of $n$) the probability of the intersection of these events,
and since they are independent we may consider each event separately. 
In order to control the probability of $E_{1,3}$ and ${E_{2, 3}}$, we use the next lemma.
\begin{lemma}\label{lemmaL}Consider a random polynomial $R_{\ell,n}$ of the form:
\be R_{\ell,n}(z)=\sum_{k=\ell+1}^na_kn^{-k/2}\rho^k z^k\ee
where $a_{\ell+1}, \ldots, a_n$ are independent and distributed as $a_k\sim N_\C(0, {n\choose k})$. 
Then there exists $c>0$ and an integer $L>0$ such that for all $n \geq \ell \geq L$ we have:
\be \PP\left\{\|R_{\ell,n}\|_{C^1(D, \C)}\leq \delta/2\right\}>c.\ee
\end{lemma}
\begin{proof}
Since $R=R_{\ell,n}$ is holomorphic, recall (from equation \eqref{eq:holder}) that for every $z$ we can write: 
\be \|JR(z)\|= \sqrt{2}|R'(z)|.\ee  In particular for $|z|\leq 1$ we can estimate:
\begin{align}\EE \|R_{\ell,n}\|_{C^1(D, \C)}&=\EE \left[\sup_{|z|\leq 1}|R(z)|+\sup_{|z|\leq 1}\|JR(z)\|\right]\\
&=\EE \left[\sup_{|z|\leq 1}|R(z)|\right]+\EE\left[\sup_{|z|\leq 1}\sqrt{2}|R'(z)|\right]\\
&\leq\left( \sum_{k=\ell+1}^n \rho^k \EE|n^{-k/2}a_k|+\rho^k \sqrt{2}k \EE|n^{-k/2}a_k| \right)\\
&\leq \left( \sum_{k=\ell+1}^n\rho^k \left(\frac{2}{\pi k!}\right)^{1/2}+\rho^k \sqrt{2} k\left(\frac{2}{\pi k!}\right)^{1/2}\right) =: M_{\ell,n} , \end{align}
where in the last step we have used the fact that 
$$\EE|n^{-k/2} a_k|\leq\left(\frac{2n!}{k!(n-k)!\pi}n^{-k}\right)^{1/2}\leq \left(\frac{2}{\pi k!}\right)^{1/2}.$$
Hence by Markov's inequality:
\be \PP\{\|R_{\ell,n}\|_{C^1(D, \C)}\leq \delta/2\}\geq 1-\frac{2\EE \|R_{\ell,n}\|_{C^1(D, \C)}}{\delta}\geq 1-\frac{2M_{\ell,n}}{\delta}.\ee
Since the series $\sum_{k=0}^\infty \rho^k \left(\left(\frac{2}{\pi k!}\right)^{1/2}+C k\left(\frac{2}{\pi k!}\right)^{1/2}\right)$ is convergent, 
when $\ell \to \infty$ the tail $M_{\ell,n}$ can be made arbitrarily small, and the statement follows.
\end{proof}
Choosing now $\ell= L$ (where $L$ is given by the previous Lemma applied to the random polynomials $P_3$ and $Q_3$) in the decompositions \eqref{eq:pdec} and \eqref{eq:qdec}, 
the Lemma implies that the probabilities of $E_{1, 3}$ and $E_{2,3}$ are each bounded from below by a nonzero constant.

For the probabilities of the remaining events we argue as follows.
Since $d$ and $L$ are fixed (they only depend on the polynomials $\alpha, \beta$), and $a_k, b_k\sim N_\mathbb{C}(0, {n\choose k})$, then the Gaussian vectors
\begin{align} (a_0, \ldots, \rho^k a_kn^{-k/2}, \ldots, \rho^d a_dn^{-d/2}), \quad\quad  (b_0, \ldots, \rho^k b_kn^{-k/2}, \ldots, \rho^d b_dn^{-d/2}),\\
(\rho^{d+1}a_{d+1}n^{-(d+1)/2}, \ldots, \rho^L a_{L}n^{-L/2}),\quad (\rho^{d+1}b_{d+1}n^{-(d+1)/2}, \ldots, \rho^L b_{L}n^{-L/2}),\end{align}
converge (as $n \to \infty$) to fixed Gaussian vectors with (nondegenerate) covariance structure.
Consequently the limits:
\be \lim_{n\to \infty}\PP\{P_1\in U_\alpha\}\quad \textrm{and}\quad \lim_{n\to \infty}\PP\{Q_1\in U_\beta\}\ee
equal the \emph{nonzero} Gaussian measures, with respect to the limit probability distribution on $\C[z]_d\simeq \C^{d+1}$,  of the two given open sets $U_\alpha$ and $U_\beta$.
Similarly the limits:
\be \lim_{n\to \infty}\PP E_{2,1}\quad\textrm{and}\quad \lim_{n\to \infty}\PP E_{2,2}\ee
both equal the \emph{nonzero} Gaussian measure (again, with respect to the limit probability distribution) of the set:
\be U=\left\{(c_{d+1}, \ldots, c_L)\in \C^{L-d}\,\,\textrm{such that}\,\, \left\|\sum_{k=d+1}^L c_k \rho^k z^k\right\|_{C^1(D, \C)}\leq \delta/2\right\}\ee
(the measure of $U$ is nonzero because it contains a non-empty open set).
\end{proof}
\subsection{The number of components of a random lemniscate}\label{sec:components}
As a consequence of Theorem \ref{thm:isotopy} and Proposition \ref{prop:components} we derive the following corollary (see Remark \ref{remark:components} for an improvement of the upper bound using Theorem \ref{thm:tangents}).
\begin{cor} There exists a constant $c_1 > 0$ such that:
\be c_1 n \leq \EE b_0(\Gamma) \leq n.\ee 
\end{cor}
\begin{proof}The upper bound $\EE b_0(\Gamma)\leq n$ follows from Proposition \ref{prop:components}.

For the lower bound we argue as follows. Let $D_n(z_1), \ldots, D_n(z_m)$ be open disjoint disks on the sphere, each one of radius $n^{-1/2}$; notice that with this choice of the radius, 
we can take $m=\lfloor cn\rfloor$ (a fraction of $n$). Let also $A=S^1\subset \R^2$ be the unit circle; 
by the previous theorem there exists $a>0$ such that for every $k=1, \ldots, m$ the probability that $(D_n(z_k),\Gamma \cap D_n(z_k))\sim (\R^2, S^1)$ is bounded below by $a$. 
In particular, for every $k=1, \ldots, m$ the curve $\Gamma$ has one component entirely contained in $D_n(z_k)$ with probability $a>0$.
Since the disks are disjoint, we thus have:
\begin{align} \EE b_0(\Gamma)&\geq \sum_{k=1}^{\lfloor cn\rfloor }\EE\{\textrm{number of components of $\Gamma$ in $D_n(z_k)$}\}\geq\sum_{k=1}^{\lfloor cn\rfloor }a\geq c_1 n,\end{align}
with $c_1 >0$.
\end{proof}
\begin{remark}[Improving the upper bound]\label{remark:components} 
As a consequence of Theorem \ref{thm:tangents} we can prove that the average number of connected components satisfies the upper bound:
\be\label{eq:improved} \EE b_0(\Gamma)\leq c_2 n + o(n) \quad \textrm{with}\quad c_2=\frac{32-\sqrt{2}}{56}\approx 0.5461...\ee
In order to prove this recall that the number of meridian tangents $\nu(\Gamma)$ 
counts the number of critical points of the map $h|_\Gamma:\Gamma\to S^1$ (defined in \eqref{eq:h}). 
Every component of $\Gamma$ not looping around the origin
has at least two meridian tangents, and hence we have:
\be \EE b_0(\Gamma)\leq \frac{1}{2} \EE \nu(\Gamma) + \EE \#\{\textrm{number of components of $\Gamma$ looping around the origin}\}.\ee
The number of components of $\Gamma$ looping around the origin can be estimated by the number of points of intersection of $\Gamma$ with 
an arc of a great circle going from the origin to infinity, 
and using the Integral Geometry Formula and Theorem \ref{thm:length} the average number of such points can be estimated by $O(\sqrt{n})$.
This is thus a lower order term and \eqref{eq:improved} follows from the asymptotic for $\EE \, \nu(\Gamma)$ provided by Theorem \ref{thm:tangents}.
\end{remark}

\section{Appendix: Comparison with real plane Kostlan curves}\label{sec:Kostlan}

A real plane \emph{Kostlan} curve is a random algebraic curve in the real projective plane $\R\textrm{P}^2$.
The random defining polynomial is built using coefficients that are independent real Gaussians with multinomial variances:
\be\label{eq:Kostlan} p(x,y,z)=\sum_{a+b+c=n}c_{abc}x^{a}y^bz^c, \quad \textrm{with}\quad c_{abc}\sim N_\R\left(0,\frac{n!}{a!b!c!}\right) .\ee
This model is also invariant under an orthogonal change of coordinates and has been studied in \cite{FLL, GaWe2, GaWe3, Sarnak}. 
Alternatively one can consider the curve $C$ defined by the same equation on the sphere $S^2\subset \R^3:$
\begin{figure}
\begin{center}
\includegraphics[width=0.28\textwidth]{RandLem100.pdf} \hspace{.2in}
\includegraphics[width=0.27\textwidth]{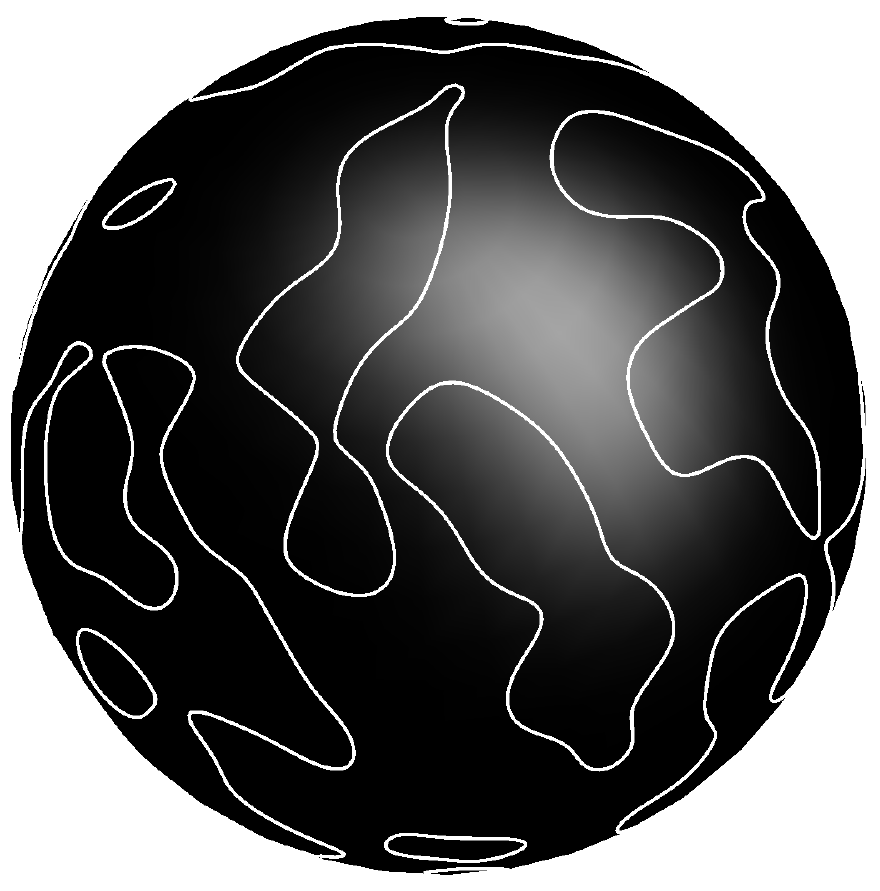}
\caption{\label{fig:LemKost} Left: Random lemniscate $n=100$.  Right: Kostlan random curve $n=100$.}
\end{center}
\end{figure}
\be C=\{(x,y,z)\in S^2\, :\, p(x,y,z)=0\}.\ee
\begin{prop}The average number of meridian tangents of a Kostlan random curve (on $S^2$) equals:
\be \EE \nu_{K}=\frac{4 \sqrt{2}}{\pi}(n(n-1))^{1/2}.\ee
\end{prop}
\begin{proof}(Sketch) The proof proceed in the exact same way as for the proof of Theorem \ref{thm:tangents}, except that we use now the stereographic projection from the sphere $S^2\subset \R^3$ to the plane $\{z=1\}$. We thus obtain:
\be \EE \nu_K=\int_{0}^{2\pi}\int_{-\pi}^\pi \sin\phi \left(\int_{\R^2}|h_1h_2|\rho_K(0,0,h_1, h_2)dh_1dh_2\right)d\phi\, d\theta,\ee
where now $\rho_K$ is the joint density of $V=(g(0,0,1), \partial_xp(0,0,1), \partial_yp(0,0,1), \partial_x^2p(0,0,1)).$ On the other hand:
\be V=(c_{0,0,0}, c_{1, 0, n-1}, c_{0,1, n-1}, 2 c_{2,0,n-2})\ee
and since the entries of $V$ are independent Gaussians with mean and variances given by \eqref{eq:Kostlan}, we see that:
\be \rho_K(0,0,h_1, h_2)=\frac{1}{2\pi \sqrt{n}}e^{-\frac{h_1^2}{n}}e^{-\frac{h_2^2}{n(n-1)}}.\ee
Consequently the result follows from the elementary evaluation:
\be \int_{\R^2}|h_1h_2|\rho_K(0,0,h_1, h_2)dh_1dh_2=\frac{\sqrt{2}}{\pi^2}(n(n-1))^{1/2}.\ee
\end{proof}


\begin{thebibliography}{9}
\bibitem{AdlerTaylor} R. J. Adler, J. E. Taylor, \emph{Random Fields and Geometry} (Springer Monographs in Mathematics), Springer, New York, 2007.

\bibitem{Ayoub}
R. Ayoub, \emph{The lemniscate and Fagnano's contributions to elliptic integrals},
Arch. Hist. Exact Sci., 29 (1984), 131-149.

\bibitem{Catanese2}
I. Bauer, F. Catanese, \emph{Generic lemniscates of algebraic functions}, Math. Ann., 307 (1997), 417-444.

\bibitem{Bell}
S.R. Bell, \emph{A Riemann surface attached to domains in the plane and complexity in potential theory}, Houston J. Math., 26 (2000), 277-297.

\bibitem{BCSS} L. Blum, P. Cucker, M. Shub and S. Smale, \emph{Complexity and real computation}, Springer-Verlag New York, 1998.

\bibitem{BCR} J. Bochnak, M. Coste, M-F. Roy: \emph{Real Algebraic Geometry}, Springer-Verlag, 1998. 

\bibitem{BogSch} E. Bogomolny, C. Schmit: \emph{Percolation model for nodal domains of chaotic wave functions}, Phys. Rev. Lett. 88 (2002).

\bibitem{BorceaShapiro}
J. Borcea, B. Shapiro, \emph{Classifying real polynomial pencils}, Int. Math. Res. Not. 2004, no. 69, 3689–3708.

\bibitem{Catanese1}
F. Catanese, M. Paluszny, \emph{Polynomial-lemniscates, trees and braids}, Topology, 30 (1991), 623-640.

\bibitem{EbenKhSh}
Ebenfelt, P., Khavinson, D., Shapiro, H., \emph{Two-dimensional shapes and lemniscates}, Complex analysis and dynamical systems IV. Part 1, 
Contemp. Math. Amer. Math. Soc. 553 (2011), 45-59.

\bibitem{Erdos}
P. Erd\"os, \emph{Some Unsolved Problems}, Combinatorics, geometry and probability (Cambridge, 1993), pp. 1-10, Cambridge Univ. Press, Cambridge, 1997.

\bibitem{EHP}
P. Erd\"os, F. Herzog, and G. Piranian, \emph{Metric  properties  of  polynomials}, J.  Anal.  Math. 6 (1958), 125-148.

\bibitem{EremenkoHayman}
A. Eremenko, W. Hayman, \emph{On the length of lemniscates}, Michigan Math. J. 46 (1999), 409-415.

\bibitem{EremJakNad}
A. Eremenko, D. Jakobson, N. Nadirashvili, \emph{On nodal sets and nodal domains on $S^2$ and $\R^2$},  
Ann. Inst. Fourier (Grenoble) 57 (2007), no. 7, 2345-2360. 

\bibitem{FryNaz} 
A. Fryntov, F. Nazarov, \emph{New estimates for the length of the Erd\"os-Herzog-Piranian lemniscate}, 
Linear and Complex Analysis 226 (2008), 49-60.

\bibitem{FLL} 
Y. Fyodorov, A. Lerario and E. Lundberg, \emph{On the number of connected components of random algebraic hypersurfaces}, 
Geometry and Physics, 95 (2015), 1-20. 

\bibitem{GaWe2}
D. Gayet, J-Y. Welschinger, 
\emph{Betti numbers of random real hypersurfaces and determinants of random symmetric matrices}, arXiv:1107.2288v1. 

\bibitem{GaWe3}
D. Gayet, J-Y. Welschinger, 
\emph{Lower estimates for the expected Betti numbers of random real hypersurfaces},  arXiv:1303.3035.

\bibitem{Grad}
I. S. Gradshteyn, I. M. Ryzhik, \emph{Table of integrals, series, and products}, 
Academic Press (English translation edited by A. Jeffrey), 1965.

\bibitem{GorMacph}
M. Goresky and R. MacPherson, \emph{Stratified Morse theory},
Springer-Verlag, Berlin, Heidelberg, New York, 1988, xiv + 272 pp.

\bibitem{GH} 
P. Griffiths, J. Harris, \emph{Principles of Algebraic Geometry}, John Wiley and Sons, 1978.

\bibitem{GustPSS} 
B. Gustafsson, M. Putinar, E.B. Saff, N. Stylianopoulos,
\emph{Bergman polynomials on an archipelago: Estimates, zeros and shape reconstruction},
Advances in Mathematics, 222 (2009), 1405-1460.

\bibitem{Hatcher} 
A. Hatcher, \emph{Algebraic Topology}, 
Cambridge University Press, 2002.

\bibitem{Hirsch} 
M. W. Hirsch, \emph{Differential topology}, 
Springer, New York, 1976.

\bibitem{Howard} 
R. Howard, \emph{The kinematic formula in Riemannian homogeneous spaces}, 
Mem. Amer. Math. Soc. 106 (1993), no. 509, vi+69.

\bibitem{JeongTan}
M. Jeong, M. Taniguchi, \emph{Bell representations of finitely connected planar domains}, 
Proc. Amer. Math. Soc., 131 (2003), 2325-2328.

\bibitem{Kharlamov}
V. Kharlamov, A. Korchagin, G. Polotovskiĭ, O. Viro (Editors), \emph{Topology of real algebraic varieties and related topics}, 
American Mathematical Society Translations--Series 2, Advances in the Mathematical Sciences, 1996.

\bibitem{KhavLS}
D. Khavinson, S-Y. Lee, A. Saez, \emph{Zeros of harmonic polynomials, critical lemniscates and caustics}, preprint, arXiv:1508.04439.

\bibitem{KhavMPT}
D. Khavinson, M. Mineev-Weinstein, M. Putinar, and R. Teodorescu, \emph{Lemniscates do not survive Laplacian growth}, 
Math. Res. Letters, 17 (2010), 335-341. 

\bibitem{Lawrence}
J.D. Lawrence, \emph{A Catalog of Special Plane Curves}, Dover,  New York, 1972.

\bibitem{LeeMak}
S-Y. Lee, N. Makarov, 
\emph{Topology of quadrature domains}, J. Amer. Math Soc., to appear.



\bibitem{Lerario2012}
A. Lerario,
\emph{Random matrices and the expected topology of quadric hypersurfaces}, 
Proc. Amer. Math. Soc. 143 (2015), 3239-3251.

\bibitem{LeLu2} 
A. Lerario, E. Lundberg, \emph{Gap probabilities and Betti numbers of a random intersection of quadrics}, 
Discrete and Computational Geometry, to appear, DOI: 10.1007/s00454-015-9741-7

\bibitem{LLstatistics} 
A. Lerario, E. Lundberg, \emph{Statistics on Hilbert's Sixteenth Problem}, 
Int. Math. Res. Not. 2015, 4293-4321.

\bibitem{LLtrunc}
A. Lerario, E. Lundberg, 
\emph{On the zeros of random harmonic polynomials: the truncated model}, preprint, arXiv:1507.01041

\bibitem{LundTotik}
E. Lundberg, V. Totik, 
\emph{Lemniscate growth}, Analysis and Mathematical Physics 3 (2013), 45-62.

\bibitem{mat} 
J. Mather, \emph{Notes on topological stability}, (1970): {\verb www.math.princeton.edu/facultypapers/mather/ }

\bibitem{Milnor}
J. Milnor, \emph{Dynamics in One Complex Variable}, 3rd Ed., Princeton University Press, 2006.

\bibitem{NagyTotik}
B. Nagy, V. Totik, \emph{Sharpening of Hilbert's lemniscate theorem}, J. d'Analyse Math., 96 (2005), 191-223.

\bibitem{NazarovSodin} F. Nazarov,  M. Sodin: \emph{On the number of nodal domains of random spherical harmonics}, 
Amer. J. Math. 131 (2009), 1337-1357.

\bibitem{NazarovSodin2}
F. Nazarov, M. Sodin, 
\emph{Asymptotic laws for the spatial distribution and the number of connected components of zero sets of Gaussian random functions}, preprint, arXiv:1507.02017 

\bibitem{PoulRans}
S. Pouliasis, T. Ransford, \emph{On the harmonic measure and capacity of rational lemniscates}, Potential Analysis, to appear.

\bibitem{SaffTotik}
E.B. Saff, V. Totik, \emph{Logarithmic potentials with external fields}, Springer, Berlin, New York, 1997.

\bibitem{Sarnak}
P. Sarnak: \emph{Letter to B. Gross and J. Harris on ovals of random plane curves},
(2011) available at: {\verb http://publications.ias.edu/sarnak/section/515 }

\bibitem{SarnakWigman}
P. Sarnak, I. Wigman:\emph{Topologies of nodal sets of random band limited functions}, preprint, arXiv:1312.7858

\bibitem{TrefBau}
L.N. Trefethen, D. Bau, \emph{Numerical Linear Algebra}, SIAM, 1997.

\bibitem{TrefEmbree}
L.N. Trefethen, M. Embree, \emph{Spectra and Pseudospectra, The Behavior of Nonnormal Matrices and Operators}, Princeton University Press, 2005.

\bibitem{Turin}
G.L. Turin, \emph{The characteristic function of Hermitian quadratic forms in complex normal variables},
Biometrika 47 (1960), 199-201.

\bibitem{Walsh}
J.L. Walsh, \emph{The location of critical points of analytic and harmonic functions}, 
American Mathematical Society, 1950.

\bibitem{Wilson}
G. Wilson: \emph{Hilbert's sixteenth problem}, Topology 17 (1978), 53-74.

\bibitem{Younsi}
M. Younsi, \emph{Shapes fingerprints and rational lemniscates}, 
Proc. Amer. Math. Soc., to appear.

\bibitem{ZeZe}
O. Zeitouni, S. Zelditch, \emph{Large deviations of empirical zero point measures on Riemann surfaces, I: $g = 0$},
Int. Nat. Math. Res. Not., to appear.



\end{thebibliography}
\end{document}